\documentclass[preprint,12pt]{article}
\usepackage[top=1.1in, bottom=1.2in, left=0.8in, right=0.8in]{geometry}
\usepackage{amsmath,amsfonts,amssymb,amsthm}
\usepackage{mathrsfs,dsfont}
\usepackage{graphicx}
\usepackage{colortbl,dcolumn}
\usepackage{psfrag}
\usepackage{booktabs}
\usepackage{cite}
\usepackage{url}
\usepackage{comment}
\usepackage{subfigure}
\usepackage{hyperref}
\allowdisplaybreaks[4]
\numberwithin{equation}{section}

\newcommand{\R}{\mathbb{R}}

\newcommand{\diff}{{\,\rm{d}}}

\newtheorem{theorem}{Theorem}[section]

\newtheorem{lemma}[theorem]{Lemma}
\newtheorem{proposition}[theorem]{Proposition}
\newtheorem{corollary}[theorem]{Corollary}
\newtheorem{remark}[theorem]{Remark}

\newtheorem{assumption}[theorem]{Assumption}

\begin{document}
\title{Strong convergence rates of stochastic theta Milstein methods for index-1 stochastic differential algebraic equations under non-globally Lipschitz conditions\footnotemark[1]}

\footnotetext{\footnotemark[1] This work was supported by National Natural Science Foundation of China (Nos. 11961029, 12671472), Yunnan Fundamental Research Projects (No. 202601AT070161) and Jiangxi Provincial Natural Science Foundation (No. 20242BAB23004).}

\author{Caiyuan Zhu\footnotemark[2], \quad
Shiye Wan\footnotemark[2], \quad
Ziheng Chen\footnotemark[2],\quad
Lin Chen\footnotemark[3]}

\footnotetext{\footnotemark[2] School of Mathematics and Statistics, Yunnan University,
Kunming, Yunnan, 650500, China.}

%

\footnotetext{\footnotemark[3] School of Statistics and Data Science, Jiangxi University of Finance and Economics, Nanchang, 330013, China. Email: chenlin@jxufe.edu.cn. Corresponding author.}

\date{}

\maketitle

\begin{abstract}
      {\rm\small This paper studies the strong convergence order of structure-preserving stochastic theta Milstein methods for a class of index-$1$ stochastic differential algebraic equations (SDAEs) with time-dependent singular matrices and non-globally Lipschitz coefficients. The singular matrix is allowed to vary in time while preserving a fixed differential algebraic splitting, and the drift and diffusion coefficients may exhibit superlinear growth. By exploiting the index-$1$ algebraic-differential decomposition of the exact solution, we identify the Milstein coefficient of the reduced stochastic differential equation directly in the original SDAE variables and establish the well-posedness and constraint preserving property of the proposed method for $\theta\in[1/2,1]$. Under a coupled monotonicity condition and suitable polynomial regularity assumptions, the method is proved to preserve the algebraic constraints at all time levels and to converge with strong order one in the root mean square norm. Numerical experiments confirm the structure-preserving property and the theoretical convergence order.} \\


\end{abstract}

\section{Introduction}

Stochastic differential equations (SDEs) provide a standard framework for modelling dynamical systems driven by random perturbations; see, for example, \cite{MR2295424,MR2380366,MR2723480}. In many applications, however, the state is also required to satisfy conservation, geometric, or other algebraic constraints. Combining these constraints with stochastic differential dynamics leads to stochastic differential algebraic equations (SDAEs) \cite{MR2450755,suthar2023explicitlyconstrainedstochasticdifferential}. The algebraic component changes both the analytical and numerical problems: a numerical method must approximate the stochastic dynamics and maintain compatibility with the constraint. General formulations, properties, and applications of SDAEs can be found in \cite{MR2990582,MR4839307,MR1998346,MR2683311,MR2266721,MR3241007} and the references therein.

In this paper, we consider the nonlinear SDAE
\begin{equation}\label{eq:SDAE}
      A_t \diff{X_t} 
      = 
      F(t,X_t) \diff{t}+G(t,X_t) \diff{W_t},
      \quad t \in (0,T]
\end{equation}
with a deterministic initial value $X_0$. Here, $\{W_t\}_{t\in[0,T]}$ is an $m$-dimensional standard Brownian motion, and $F\colon[0,T]\times\mathbb R^d\to\mathbb R^d$ and $G\colon[0,T]\times\mathbb R^d\to\mathbb R^{d\times m}$ are the drift and diffusion coefficients. For every $t\in[0,T]$, the matrix $A_t\in\mathbb R^{d\times d}$ is singular. Its singular value decomposition is assumed to have fixed left and right singular subspaces. Consequently, the projectors $P=A_t^-A_t$ and $R=I-A_tA_t^-$ are time independent, where $A_t^-$ denotes the Moore--Penrose inverse. The analysis is further carried out under the index-$1$ conditions that the noise has no algebraic component and that the algebraic variables are globally and uniquely determined; the precise conditions are stated in Assumptions \ref{asm:At} and \ref{ass:index1}.

Classical numerical studies of SDAEs include methods for systems with a constant singular matrix \cite{MR1998346,MR1658742,MR3979263,MR2931358,MR3373557}. The convergence theory has subsequently been extended to non-globally Lipschitz coefficients: in particular, stochastic theta methods for index-$1$ SDAEs with time-dependent singular matrices and non-globally Lipschitz coefficients achieve strong order $1/2$ \cite{chen2025strong}. For unconstrained SDEs, classical higher-order implicit and stability-oriented Milstein-type methods have been studied in \cite{MR1149488,MR3920055,MR3857916}. Milstein-type schemes for locally or non-globally Lipschitz and highly nonlinear coefficients were developed in \cite{MR3037286,MR3082312,MR3574237}. In particular, strong order one in the root mean square norm was established for implicit Milstein-type methods with non-globally Lipschitz drift and diffusion coefficients in \cite{MR4601143}. These results do not, however, directly show that a Milstein correction formulated in the original SDAE variables is compatible with the algebraic constraint or that it yields strong order one for time-dependent singular matrices with superlinearly growing coefficients. This question is also relevant to multilevel Monte Carlo simulation, whose efficiency depends on the strong accuracy of the underlying path approximation \cite{MR2479233}.

Let $\Delta=T/K$ for $K\in\mathbb N$ and $t_k=k\Delta$. In this work, for $\theta\in[\frac12,1]$, we propose the stochastic theta Milstein method defined by $x_0=X_0$ and
\begin{align}\label{eq:ABEM}
      A_{t_{k}}x_{k+1}
             =&~
       A_{t_k}x_k +\theta \big(R+A_{t_k}A_{t_{k+1}}^-\big)F(t_{k+1},x_{k+1}) \Delta
      +
      (1-\theta) F(t_{k},x_{k}) \Delta \notag
      \\&~+
      G(t_k,x_k) \Delta W_{k}
      +
      \sum_{j_1,j_2=1}^m \widetilde{\mathcal{L}}^{j_1} G_{j_2}(t_k, x_k)  I_{j_1,j_2}^{t_k,t_{k+1}}
\end{align}
for $k=0,1,\cdots,K-1$, where $\Delta W_k=W_{t_{k+1}}-W_{t_k}$ and $G_j(t,x)$ is the $j$th column of $G(t,x)$. With $J(t,x):=A_t+R\partial_xF(t,x)$, the SDAE-level Milstein coefficient is
\begin{equation*}
      \widetilde{\mathcal{L}}^{j_1}G_{j_2}(t,x) := \partial_xG_{j_2}(t,x)J(t,x)^{-1}G_{j_1}(t,x), \quad j_1,j_2=1,2,\cdots,m,
\end{equation*}
and
\begin{equation*}
      I_{j_1, j_2}^{t_k, t_{k+1}} := \int_{t_k}^{t_{k+1}}\int_{t_k}^{s_2} \diff W_{s_1}^{j_1}\,\diff W_{s_2}^{j_2}.
\end{equation*}
The explicit case $\theta=0$ is not considered, since the singularity of $A_{t_k}$ prevents an explicit update from uniquely determining the algebraic component of $x_{k+1}$, and the constraint at the new time level is not automatically enforced. Thus an implicit contribution with $\theta>0$ is essential for both well-posedness and constraint preservation. The further restriction $\theta\geq1/2$ is used in the mean square stability argument, where it makes the explicit drift-square contribution nonpositive. A central ingredient of our analysis is the index-$1$ algebraic-differential decomposition. At the continuous level, the exact solution can be written as $X_t=U_t+\hat V(t,U_t)$, where $\hat V$ is the algebraic reconstruction map and the differential component $U_t=PX_t$ satisfies an inherent SDE. At the discrete level, however, we do not discretize this inherent SDE directly. Instead, the stochastic theta Milstein method is formulated in the original SDAE variables, and we extract the induced Milstein dynamics of its differential component. The factor $R+A_{t_k}A_{t_{k+1}}^-$ produces the correct new-time implicit drift after projection, while the SDAE-level Milstein term induces precisely the Milstein correction of the inherent SDE. This viewpoint connects the constraint preserving SDAE discretization with the strong error analysis of the induced differential approximation.

Since $A_{t_k}$ is singular, neither the solvability of the implicit update nor the preservation of the algebraic constraint is automatic. Using the fixed projector structure and the reconstruction map, we reduce the SDAE-level update to a monotone implicit equation for the differential component $u_{k+1}=Px_{k+1}$ and then recover the algebraic component as $\hat V(t_{k+1},u_{k+1})$. This yields a unique adapted numerical solution for sufficiently small stepsizes and ensures that $x_k$ lies on the constraint manifold at every time level. To establish the strong convergence order, we derive an error-reduction estimate that bounds the global mean square error in terms of the local residual and its conditional expectation. The local residual is analyzed through drift and diffusion expansions of the inherent SDE, with the iterated stochastic integrals cancelling the leading diffusion term. Polynomial moment bounds, H\"{o}lder continuity and derivative estimates then give
\begin{equation*}
      \mathbb E[|\mathcal R_k|^2]\leq C\Delta^3, \quad \mathbb E\big[|\mathbb E[\mathcal R_k\mid\mathcal F_{t_{k-1}}]|^2\big] \leq C\Delta^4.
\end{equation*}
Combining these local estimates with the error-reduction argument yields strong convergence of order one in the root mean square norm.

The main contributions of this paper are summarized as follows. First, we formulate the stochastic theta Milstein method directly in the original SDAE variables and identify its Milstein correction with that of the inherent SDE. Second, we prove that the proposed method is well posed for sufficiently small stepsizes and preserves the algebraic constraint at every time level. Third, by combining an error-reduction argument with local Milstein expansions and polynomial moment estimates, we establish strong convergence of order one under a coupled monotonicity condition and polynomial regularity assumptions. Numerical experiments confirm both the theoretical convergence order and the constraint preserving property.

The rest of this paper is organized as follows. Section \ref{SDAEs} introduces the index-$1$ structure, the algebraic reconstruction, the reduced SDE, and the regularity estimates used in the analysis. Section \ref{sec:orderAt} proves the well-posedness and constraint preservation of the stochastic theta Milstein method, derives the local residual bounds, and establishes the global convergence rate. Finally, Section \ref{sec:experiments} presents numerical experiments for the convergence order and the constraint preserving property.

\section{Index-1 SDAE setting and reduced SDE}\label{SDAEs}

In this section, we introduce the structural assumptions on the index-$1$ SDAE \eqref{eq:SDAE} and derive the associated reduced dynamics for its differential component. We begin with some notation.  Let $\langle \cdot,\cdot \rangle$ and $|\cdot|$ denote the Euclidean inner product and norm in $\R^d$, respectively. For any matrix $B\in\R^{d\times m}$, $B^\top$ denotes its transpose, $|B|:=\sqrt{\operatorname{trace}(B^\top B)}$ its Frobenius norm, $\operatorname{Im}(B):=\{Bx:x\in\R^m\}$ its image, and $\operatorname{Ker}(B):=\{x\in\R^m:Bx=0\}$ its kernel. For a differentiable vector-valued function $H$, $\partial_xH$ denotes its Jacobian matrix with respect to $x$; analogous notation is used for other state variables. Moreover, the letter $C$ denotes a generic positive constant whose value may change from line to line but is independent of the stepsize $\Delta$.

\subsection{Matrix structure and constraint manifold}

We first make the following assumption on the singular matrix $A_{t}$ for $t \in [0,T]$.

\begin{assumption}\label{asm:At}
      For any $t \in [0,T]$, the singular value decomposition of matrix $A_t$ takes the form $A_t = M\Sigma_t N$, where $M,N \in \R^{d \times d}$ are orthogonal matrices and $\Sigma_t = \operatorname{diag}\bigl(\sigma_1(t), \cdots, \sigma_r(t), 0, \cdots, 0\bigr)$ is continuously differentiable with $r \in \{1,2,\cdots,d-1\}$. Moreover, for every $i \in \{1,2,\cdots,r\}$, $\sigma_i \in C^2([0,T];\R)$ and there exist constants $\underline\sigma, \overline\sigma, C_\sigma>0$ such that
      \begin{align*}
            0 < \underline{\sigma}
            \leq
			\sigma_{i}(t)
			\leq
			\overline{\sigma} < \infty,
            \quad
            |\sigma_i'(t)|+|\sigma_i''(t)|
            \leq C_\sigma,
            \quad t \in [0,T].
     \end{align*}
\end{assumption}

The Moore--Penrose inverse of any real matrix exists and is unique; see, e.g., \cite{MR4759037}. We denote by $A_t^-$ the Moore--Penrose inverse of $A_t$. Assumption \ref{asm:At} implies
\begin{equation*}
      A_t^-=N^\top\Sigma_t^-M^\top, \quad \Sigma_t^-=\operatorname{diag}\bigl(\sigma_1(t)^{-1},\cdots, \sigma_r(t)^{-1},0,\cdots,0\bigr).
\end{equation*}
Consequently, the time independent operator
\begin{equation*}
      P:=A_t^-A_t
      =N^\top
      \begin{pmatrix}I_r&0\\0&\mathbf 0_{d-r}\end{pmatrix}N
\end{equation*}
is the orthogonal projector onto the differential subspace. The complementary projector $Q:=I-P$ projects onto $\operatorname{Ker}(A_t)$ and satisfies $A_tQ=0$. Moreover, $R:=I-A_tA_t^-$ is the orthogonal projector onto the left null space of $A_t$ and satisfies $RA_t=A_t^-R=0$. These projectors enable us to split the solution of \eqref{eq:SDAE} into differential and algebraic components:
\begin{equation}\label{eq:PQ}
      X_t = PX_t+QX_t := U_t+V_t,
      \quad
      U_t \in \operatorname{Im}(P), \ V_t\in \operatorname{Im}(Q).
\end{equation}

\begin{assumption}\label{ass:index1}
      Let the initial value $X_0$ be deterministic and satisfy $RF(0,X_0) = 0$.
      Besides, the algebraic Jacobian $J(t,x) := A_t+R \partial_xF(t,x)$ is invertible for each $(t,x) \in [0,T] \times \R^{d}$ and there exists a constant $L_{J} > 0$ such that $\sup_{(t,x) \in [0,T] \times \R^{d}} |J(t,x)^{-1}| \leq L_{J}$. Moreover, the diffusion coefficient is compatible with the algebraic constraint, namely,
            \begin{equation} \label{eq:noise-compatibility}
            RG(t,x)=0, \quad (t,x)\in[0,T]\times\mathbb R^d. \end{equation}
            
\end{assumption}

For every $(t,x)\in[0,T]\times\R^d$, the algebraic equation
\begin{equation}\label{eq:v(t,x)}
      A_tv+RF(t,x+v)=0
\end{equation}
admits a unique solution $v=\hat V(t,x)$; see, e.g., \cite{MR4915812,MR1998346}.

\subsection{Assumptions on non-globally Lipschitz coefficients}

We next state the coefficient conditions used for \eqref{eq:SDAE}. Let $F\colon[0,T]\times\R^d\to\R^d$ and $G=(G_1,\ldots,G_m)\colon[0,T]\times\R^d\to\R^{d\times m}$ be the drift and diffusion coefficients, respectively, where $G_j(t,x)\in\R^d$ denotes the $j$-th column of $G(t,x)$ for $j=1,\cdots,m$.

\begin{assumption}\label{asm:FG}
      There exist constants $q>1$, $p_1>8$, $L_1>0$, and $\Delta_0\in(0,T]$ such that 
      for all $x,y \in \mathbb{R}^{d}$ and $s,t \in [0,T]$, 
      \begin{align}\label{asm:FG-FG}
            &~2\big\langle Px-Py, 
			A_t^- F(t,x) - A_t^- F(t,y) \big\rangle 
			+
			(p_1-1)|A_t^- G(t,x) - A_t^- G(t,y)|^2 \notag
			\\&~+
            \frac{q}{2}\Delta_0\sum_{j_1,j_2=1}^m 
            \big|A_t^-\widetilde{\mathcal{L}}^{j_1}G_{j_2}(t,x) 
            - A_t^-\widetilde{\mathcal{L}}^{j_1}G_{j_2}(t,y)\big|^2
            \leq 
			L_1|x-y|^2.
      \end{align}
\end{assumption}

\begin{assumption}\label{ass:time-regularity}
      Assume that $F$ and $G_j$, $j=1,\cdots,m$, belong to $C^{1,2}([0,T]\times\mathbb R^d)$.
      There exists constants $\gamma\in [1,\tfrac{p_1+2}{10})$ and $C > 0$, such that for all $x,y\in\mathbb R^d$ and $s,t\in[0,T]$, 
      \begin{equation}\label{asm:FG-absIa}
			\left|  \partial_xF(t,x) - 
            \partial_xF(s,y)  \right| 
            \leq C\big((1+|x|+| y|)^{\gamma-2} | x 
            - y |+(1+|x|+|y|)^{\gamma}|t-s|\big), 
      \end{equation}
      \begin{equation}\label{asm:FG-absJa}
			\left| \partial_tF(t,x) - 
            \partial_tF(s,y) \right| 
            \leq C\big((1+|x|+|y|)^{\gamma-1}
            |x-y|
            +
            (1+|x|+|y|)^{\gamma}|t-s|\big),
      \end{equation}
      \begin{equation}\label{asm:FG-absKa}
			\left|  \partial_xG_j(t,x) 
            - \partial_xG_j(s,y)\right|^2
			\leq C\big((1+|x|+|y|)^{\gamma-3} |x -y|^2+(1+|x|+|y|)^{\gamma+1}|t-s|^2\big), 
      \end{equation}
      \begin{equation}\label{asm:FG-absMa}
			\left|\partial_tG_j(t,x) 
            - \partial_tG_j(s,y)\right|^2 
			\leq C\big((1+|x|+|y|)^{\gamma-1} | x 
			- y |^2+(1+|x|+|y|)^{\gamma+1}|t-s|^2\big).
      \end{equation}
\end{assumption}

\begin{remark}\label{rm:FG}
Let $t=s$ and $y=0$, it follows from Assumption \ref{ass:time-regularity} that
      \begin{equation}\label{asm:dF}
			|\partial_x F(t,x)|
			\leq
			C (1+|x|)^{\gamma-1},  \ \ \ 
    		|\partial_t F(t,x)|
			\leq
			C (1+|x|)^{\gamma}        
      \end{equation}
      \begin{equation}\label{asm:dG}
			|\partial_x G(t,x)|^2
			\leq
			C (1+|x|)^{\gamma-1},  \ \ \ 
    		|\partial_t G(t,x)|^2
			\leq
			C (1+|x|)^{\gamma+1} .
      \end{equation}
      Moreover the mean value theorem yields that
      \begin{equation}\label{asm:FG-absFa}
			|F(t,x)-F(s,y)|
			\leq
			C\big((1+|x|+|y|)^{\gamma-1} |x-y|
			+ (1+|x|+|y|)^{\gamma} |t-s|\big),
      \end{equation}
      \begin{equation}\label{asm:FG-absG}
			|G(t,x)-G(s,y)|^{2}
			\leq
			C\big((1+|x|+|y|)^{\gamma-1} |x-y|^{2}
			+ (1+|x|+|y|)^{\gamma+1} |t-s|^2\big),
      \end{equation}
      \begin{equation}\label{asm:bound_FG}
			|F(t,x)|
			\leq
			C (1+|x|)^{\gamma},
            \ \ \ 
    		|G(t,x)|^2
			\leq
			C (1+|x|)^{\gamma+1} .       
      \end{equation}
\end{remark}     

We first record the existence, uniqueness, and constraint property of the exact solution.

\begin{proposition}\label{prop:exact-solution-constraint}
Suppose that Assumptions \ref{asm:At}, \ref{ass:index1}, \ref{asm:FG} and \ref{ass:time-regularity} hold. Then the SDAE \eqref{eq:SDAE} admits a unique continuous adapted solution $\{X_t\}_{t\in[0,T]}$, which satisfies
      \begin{equation}\label{eq:sdae_integral}
            \int_{0}^{t} A_s\,dX_{s}
            =
            \int_{0}^{t} F(s,X_{s})\,ds
            +
            \int_{0}^{t} G(s,X_{s})\,dW_s,
            \quad t \in [0,T].
      \end{equation}
      Moreover, the exact solution $\{X_{t}\}_{t \in [0,T]}$ satisfies the algebraic constraint
      \begin{equation*}
            X_{t} \in \mathcal{M}_t, \quad t \in [0,T], \quad \mathbb{P}\text{\rm{-a.s.}},
      \end{equation*}
      where $\mathcal{M}_{t} := \{x \in \R^{d} : RF(t,x) = 0\}$ for each $t \in [0,T]$.  
\end{proposition}

\begin{proof}
      From \eqref{asm:FG-absFa} and \eqref{asm:FG-absG}, $F$ and $G$ are locally Lipschitz in $x$, uniformly in $t\in[0,T]$. Moreover, setting $y=0$ in \eqref{asm:FG-FG} and discarding the nonnegative Milstein term give
      \begin{align*}
            &\big\langle Px,A_t^-F(t,x)-A_t^-F(t,0)\big\rangle
            +\frac{p_1-1}{2}\big|A_t^-G(t,x)-A_t^-G(t,0)\big|^2
            \leq \frac{L_1}{2}|x|^2.
      \end{align*}
      Since $p_1>2$, the weighted Young inequality with $\varepsilon = p_1-2 > 0$ together with $|Px|\leq|x|$ yields
      \begin{align*}
            &\big\langle Px,A_t^-F(t,x)\big\rangle
            +\frac12\big|A_t^-G(t,x)\big|^2
            \\
            =&~
            \big\langle Px,A_t^-F(t,x)-A_t^-F(t,0)\big\rangle
            +\big\langle Px,A_t^-F(t,0)\big\rangle
            \\
            &~+\frac12\big|
            A_t^-G(t,x)-A_t^-G(t,0)+A_t^-G(t,0)
            \big|^2
            \\
            \leq&~
            \big\langle Px,A_t^-F(t,x)-A_t^-F(t,0)\big\rangle
            +\frac{p_1-1}{2}
            \big|A_t^-G(t,x)-A_t^-G(t,0)\big|^2
            \\
            &~+\frac12|x|^2
            +\frac12|A_t^-F(t,0)|^2
            +\frac{p_1-1}{2(p_1-2)}|A_t^-G(t,0)|^2
            \\
            \leq&~
            \frac{L_1+1}{2}|x|^2
            +\frac12|A_t^-F(t,0)|^2
            +\frac{p_1-1}{2(p_1-2)}|A_t^-G(t,0)|^2.
      \end{align*}
      Assumption \ref{asm:At} implies that $A_t^-$ is uniformly bounded on $[0,T]$. Moreover, Assumption \ref{ass:time-regularity} and the compactness of $[0,T]$ imply that $F(t,0)$ and $G(t,0)$ are uniformly bounded. Consequently,
      \begin{equation}\label{eq:exact-solution-coercivity}
            \big\langle Px,A_t^-F(t,x)\big\rangle
            +\frac12\big|A_t^-G(t,x)\big|^2
            \leq C(1+|x|^2).
      \end{equation}
      Assumption \ref{asm:At} also ensures that $A_t$ is continuously differentiable and that $P=A_t^-A_t$ is constant. Assumption \ref{ass:index1} gives the index-$1$ conditions, the continuity and uniform invertibility of $J$, and the noise compatibility condition $RG(t,x)=0$, which is equivalent to $\operatorname{Im}(G(t,x))\subseteq\operatorname{Im}(A_t)$. Finally, the deterministic initial value $X_0$ belongs to $L^2(\Omega;\mathbb R^d)$. Consequently, \eqref{eq:exact-solution-coercivity} and \cite[Theorem 1]{MR4915812} imply that \eqref{eq:SDAE} admits a unique continuous adapted solution.

      It remains to verify the constraint property. Writing \eqref{eq:SDAE} in integral form, applying the time independent projector $R$, and using $RA_s=0$ and $RG(s,X_s)=0$, we obtain, on a set of probability one,
      \begin{equation}\label{eq:constraint-integral-zero}
      \int_0^t RF(s,X_s)\,\diff s=0, \quad t\in[0,T].
      \end{equation}
      For every sample point in this set, the integrand $t\mapsto RF(t,X_t)$ is continuous because $X$ has continuous sample paths and $F$ is continuous. By \eqref{eq:constraint-integral-zero} and the fundamental theorem of calculus,
      \begin{equation*}
      RF(t,X_t)=0, \quad t\in[0,T].
      \end{equation*}
      Consequently, $X_t\in\mathcal M_t$ for every $t\in[0,T]$ almost surely.
\end{proof}

\subsection{Algebraic reconstruction and reduced SDE}

By Proposition \ref{prop:exact-solution-constraint}, the exact solution belongs to $\mathcal M_t$. Since $X_t=U_t+V_t$ and $A_tV_t=A_tQX_t=0$, the defining constraint of $\mathcal M_t$ is equivalently written as
\begin{equation}\label{eq:constraint}
      A_tV_t+RF(t,U_t+V_t)=0,
      \quad t\in[0,T].
\end{equation}
Equation \eqref{eq:v(t,x)} defines the algebraic reconstruction map $(t,x)\mapsto\hat V(t,x)$. Multiplying \eqref{eq:v(t,x)} by $A_t^-$ and using $A_t^-R=0$, we obtain $P\hat V(t,x) = 0$, i.e., $\hat V(t,x) \in \operatorname{Im}(Q)$. We next establish the reconstruction formula \eqref{def:Qx}. Since $A_tQx=0$, for any $x\in\R^d$ and $t\in[0,T]$,
\begin{align*}
      A_t\big(Qx+\hat V(t,x)\big)
      +RF\big(t,Px+Qx+\hat V(t,x)\big)
      =A_t\hat V(t,x)+RF\big(t,x+\hat V(t,x)\big)
      =0.
\end{align*}
Thus $Qx+\hat V(t,x)$ solves \eqref{eq:v(t,x)} with $x$ replaced by $Px$. By the uniqueness of the solution to \eqref{eq:v(t,x)},
\begin{equation}\label{def:Qx}
      x=Px+Qx,
      \quad
      \hat V(t,Px)=Qx+\hat V(t,x).
\end{equation}
Moreover, $V_t=\hat V(t,U_t)$ is the unique solution of \eqref{eq:constraint}, and \eqref{eq:PQ} becomes
\begin{equation}\label{eq:XU}
      X_t=U_t+\hat V(t,U_t),
      \quad t\in[0,T].
\end{equation}
Since $P=A_t^-A_t$ is time independent under Assumption \ref{asm:At}, multiplying \eqref{eq:SDAE} by $A_t^-$ gives
\begin{equation*}
      P\,\text{d}X_t =A_t^-F(t,X_t)\,\text{d}t +A_t^-G(t,X_t)\,\text{d}W_t, \quad t\in[0,T].
\end{equation*}
Together with $U_t=PX_t$ and \eqref{eq:XU}, this yields the following unconstrained SDE for the differential component:
\begin{equation}\label{eq:inherent}
      U_t-U_0
      =
      \int_0^t A_s^-F\big(s,U_s+\hat V(s,U_s)\big)\,\text{d}s
      +
      \int_0^t A_s^-G\big(s,U_s+\hat V(s,U_s)\big)\,\text{d}W_s,
      \quad t\in[0,T].
\end{equation}
Here $U_0=PX_0$ and $RF(0,X_0)=0$. Consequently, solving \eqref{eq:SDAE} is equivalent to solving the coupled system \eqref{eq:constraint} and \eqref{eq:inherent}.

Define the reconstruction map and the reduced coefficients on the full space by
\begin{equation*}
      \Psi_t(u):=u+\hat V(t,u), \quad f(t,u):=A_t^-F\bigl(t,\Psi_t(u)\bigr), \quad g(t,u):=A_t^-G\bigl(t,\Psi_t(u)\bigr), \quad u\in\R^d.
\end{equation*}
Equation \eqref{def:Qx} implies, for every $t\in[0,T]$ and $u\in\R^d$, we have
\begin{equation}\label{eq:projection-invariant-extension}
      \Psi_t(u)=\Psi_t(Pu),
      \quad
      f(t,u)=f(t,Pu),
      \quad
      g(t,u)=g(t,Pu),
\end{equation}
and
\begin{equation}\label{eq:projection-invariant-extension2}
      Pf(t,u)=f(t,u),
      \quad
      Pg(t,u)=g(t,u).
\end{equation}
Here \eqref{eq:projection-invariant-extension} follow from $PA_t^-=A_t^-$, since $P=A_t^-A_t$ and $A_t^-A_tA_t^-=A_t^-$. Moreover, $\Psi_t(u)\in\mathcal M_t$ and $P\Psi_t(u)=Pu$. Thus $\Psi_t$, $f$, and $g$ are projection-invariant extensions from the differential subspace $\operatorname{Im}(P)$ to $\R^d$; the exact differential process itself still takes values in $\operatorname{Im}(P)$. With these definitions, \eqref{eq:inherent} can be rewritten as
\begin{equation}\label{eq:inherent_U}
      U_t-U_0
      =
      \int_0^t f(s,U_s)\,\text{d}s
      +
      \int_0^t g(s,U_s)\,\text{d}W_s,
      \quad t\in[0,T].
\end{equation}

\subsection{Properties of the algebraic map and reduced coefficients}

We next collect the estimates for the matrix family, the reconstruction map, and the reduced coefficients that will be used in the numerical analysis. The first lemma concerns uniform bounds for $A_t$, its Moore--Penrose inverse, and their time derivatives.

\begin{lemma}\label{lem:prop:A}
Suppose that Assumption \ref{asm:At} holds. Then there exist constants $C_A := \max_{t \in [0,T]} |\Sigma_t'| < \infty$ and $C_B:=\max_{t \in [0,T]} |\Sigma_t''| < \infty$ such that
      \begin{equation*}
            |A_t|
            \leq
            \sqrt{r}\,\overline{\sigma},
            \quad
            |A_t' |
            \leq
            C_A,
            \quad
            |A_t'' |
            \leq
            C_B,\quad t \in [0,T].
      \end{equation*}
Moreover, the map $t\mapsto A_t^-$ is twice differentiable, and there exists a constant $C>0$ such that
      \begin{equation*}
           |A_t^-|
            \leq
            \frac{\sqrt{r}}{\underline{\sigma}},
            \quad 
            \left| (A_t^-)' \right| \vee \left| (A_t^-)'' \right| \leq C,\quad t \in [0, T].
      \end{equation*}
\end{lemma}

\begin{proof}
      Using $A_t=M\Sigma_tN$ and $A_t^-=N^\top\Sigma_t^-M^\top$, the orthogonality of $M$ and $N$ gives
      \begin{equation*}
            |A_t| = |\Sigma_t|, \quad |A_t^-| = |\Sigma_t^-|, \quad |A_t'| = |\Sigma_t'|, \quad t \in [0, T].
      \end{equation*}
      The diagonal forms of $\Sigma_t$ and $\Sigma_t^-$ imply
      \begin{equation*}
            |A_t|^2 = \sum_{i=1}^r \sigma_i(t)^2 \leq r \overline{\sigma}^2, \quad |A_t^-|^2 = \sum_{i=1}^r (\sigma_i(t))^{-2} \leq \frac{r}{\underline{\sigma}^2}.
      \end{equation*}
      Thus $|A_t|\leq\sqrt r\,\overline\sigma$ and $|A_t^-|\leq\sqrt r/\underline\sigma$. Since $t\mapsto\Sigma_t'$ and $t\mapsto\Sigma_t''$ are continuous on $[0,T]$, we also have $|A_t'|\leq C_A$ and $|A_t''|\leq C_B$. Because $\sigma_i\in C^2([0,T])$ and $\sigma_i(t)\geq\underline\sigma>0$, the reciprocal $t\mapsto\sigma_i(t)^{-1}$ belongs to $C^2([0,T])$ and satisfies
      \begin{equation*}
            \left(\frac{1}{\sigma_i}\right)'(t) = -\sigma_i'(t) \sigma_i(t)^{-2},\quad\left(\frac{1}{\sigma_i}\right)''(t) = \frac{2(\sigma_i'(t))^2 - \sigma_i(t)\sigma_i''(t)}{\sigma_i(t)^3}, \quad i = 1, \dots, r.    
      \end{equation*}
      Consequently, $t\mapsto\Sigma_t^-$ is twice differentiable and satisfies
      \begin{align*}
            &~(\Sigma_t^-)' = \operatorname{diag}\left(-\sigma_1'(t)\sigma_1(t)^{-2}, \cdots, -\sigma_r'(t)\sigma_r(t)^{-2}, 0, \cdots, 0\right), \notag \\
            &~(\Sigma_t^-)''=\operatorname{diag}\left(\frac{2(\sigma_1'(t))^2 - \sigma_1(t)\sigma_1''(t)}{\sigma_1(t)^3}, \cdots, \frac{2(\sigma_r'(t))^2 - \sigma_r(t)\sigma_r''(t)}{\sigma_r(t)^3}, 0, \cdots, 0\right).\notag 
      \end{align*}
      Using $(A_t^-)'=N^\top\bigl((\Sigma_t^-)'\bigr)M^\top$ and $(A_t^-)''=N^\top\bigl((\Sigma_t^-)''\bigr)M^\top$ yields
      \begin{gather*}
            \left| (A_t^-)' \right| = \left| (\Sigma_t^-)' \right| = \left( \sum_{i=1}^r \frac{|\sigma_i'(t)|^2}{\sigma_i(t)^4} \right)^{1/2} \leq \frac{|\Sigma_t'|}{\underline{\sigma}^2} \leq \frac{C_A}{\underline{\sigma}^2},
            \\
            \left| (A_t^-)'' \right| = \left( \sum_{i=1}^r \left|\frac{2(\sigma_i'(t))^2-\sigma_i(t)\sigma_i''(t)} {\sigma_i(t)^3}\right|^2 \right)^{1/2} \leq \frac{2C_A^2}{\underline{\sigma}^3}+\frac{C_B}{\underline{\sigma}^2}.
      \end{gather*}
      This completes the proof.
\end{proof}

\begin{lemma}\label{lem-v}
Denote $\hat{L} := \Big(\sup\limits_{t \in [0,T]}|A_t|\Big)L_{J}+d$ and let $\hat V$ be the reconstruction map defined by \eqref{eq:v(t,x)}. Suppose that Assumptions \ref{asm:At}, \ref{ass:index1}, \ref{asm:FG}, and \ref{ass:time-regularity} hold. Then $\hat V\in C^{1,2}([0,T]\times\mathbb R^d;\mathbb R^d)$, and there exists a constant $C>0$ such that for any $s,t\in[0,T]$ and $u,v\in\mathbb R^d$,
      \begin{gather}
            \label{lem-v-rs1}
            \big|\hat{V}(t,u)\big|
            \leq
            C\big(1+|u|\big),
            \quad
            \big|u+\hat{V}(t,u)\big|
            \leq 
            C\big(1+|Pu|\big),
            \\\label{lem-v-rs2}
            \big|\hat{V}(t,u) - \hat{V}(s,v)\big|
            \leq
            \hat{L}\big|u-v\big| 
           +C\big(1+|u|+|v|\big)^{\gamma}
            \big|t-s\big|,
                        \\\label{lem-v-rs0}
            \big|\big(u+\hat{V}(t,u)\big)
            - \big(v+\hat{V}(s,v)\big)\big|
            \leq
            \big(1+\hat{L}\big)\big|Pu - Pv\big| 
             +C\big(1+|Pu|+|Pv|\big)^\gamma\big|t-s\big|,
           \\\label{lem-v-rs3}
            \big|\partial_t\hat V(t,u)\big|
            \leq C(1+|u|)^\gamma,
            \quad
            \big|\partial_t\hat{V}(t,u)-\partial_t\hat{V}(s,v)\big| 
            \leq 
            C(1+|u|+|v|)^{3\gamma-1}
            \big(|u-v|+|t-s|\big).
      \end{gather}
\end{lemma}

\begin{proof}
      Under the assumptions stated above, the algebraic reconstruction map $\hat{V}$ satisfies \eqref{lem-v-rs1}, \eqref{lem-v-rs2}, and \eqref{lem-v-rs0}(see \cite[Lemma 2.2]{chen2025strong}). It remains to prove \eqref{lem-v-rs3}.  
      For any $t\in[0,T], u\in\mathbb{R}^{d}$, let 
      $$B(t,u) := A_t+R\partial_xF(t,u+\hat V(t,u)),
      \quad
      D(t,u) := A_t'\hat V(t,u)+R\partial_tF(t,u+\hat V(t,u)).$$ 
      Then Assumptions \ref{asm:FG}, \ref{ass:time-regularity} together with \eqref{lem-v-rs1} yield $B(t,u) = J(t,u+\hat V(t,u))$ is nonsingular, $|B(t,u)^{-1}|\leq L_J$ and $|D(t,u)|\leq C(1+|u|)^\gamma$.
      Define
      \begin{equation*}
      h(t,u,v):=A_t v+RF(t,u+v), \quad t\in[0,T],u,v \in \mathbb{R}^{d}.
      \end{equation*}
      The asserted $C^{1,2}$ regularity follows from the parameter-dependent implicit function theorem, because $h$ has the required regularity and its derivative with respect to $v$ is the invertible matrix $J(t,u+v)$.      
      By the definition of $\hat{V}$, one has
      \begin{equation}\label{eq:h-zero}
            h\big(t,u,\hat{V}(t,u)\big)\equiv 0,
            \quad
            t\in[0,T], u\in\mathbb{R}^{d}.
      \end{equation}
      Differentiating \eqref{eq:h-zero} with respect to $t$ gives
     \begin{align*}
           0
           =&~
           \frac{\mathrm d}{\mathrm dt}h\big(t,u,\hat{V}(t,u)\big) \notag
           \\=&~
           A_t'\hat{V}(t,u) + A_t\partial_t\hat{V}(t,u)
           +
           R\big(\partial_tF\big(t,u+\hat{V}(t,u)\big)
           +
           \partial_xF\big(t,u+\hat{V}(t,u)\big)\partial_t\hat{V}(t,u)\big)
           \\=&~
           D(t,u) + B(t,u)\partial_t\hat{V}(t,u).
     \end{align*}
     The fact $|B(t,u)^{-1}|\leq L_J$ and $|D(t,u)|\leq C(1+|u|)^\gamma$ give
      \begin{align*}
            \big|\partial_t\hat V(t,u)\big| 
            =
            \left|- B(t,u)^{-1} D(t,u) \right| \nonumber 
            \leq
            \left|B(t,u)^{-1}\right| \left| D(t,u) \right| \nonumber
            \leq
            C\big(1+|u|^{\gamma}\big),
      \end{align*}
      and
      \begin{align*}
            &~\big|\partial_t\hat{V}(t,u)-\partial_t\hat{V}(s,v)\big|
            =
            \big|-B(t,u)^{-1}D(t,u)+B(s,v)^{-1}D(s,v)\big|\\
            =&~
            \big|B(s,v)^{-1}D(s,v)-B(s,v)^{-1}D(t,u)
            +B(s,v)^{-1}D(t,u)-B(t,u)^{-1}D(t,u)\big|
            \\\leq&~
            \big|B(s,v)^{-1}D(s,v)-B(s,v)^{-1}D(t,u)\big|
            +
            \big|B(s,v)^{-1}D(t,u)-B(t,u)^{-1}D(t,u)\big|
            \\\leq&~
            |B(s,v)^{-1}| |D(s,v)-D(t,u)|+|B(s,v)^{-1} -B(t,u)^{-1}|\,|D(t,u)|\\
            =&~
            |B(s,v)^{-1}|\,|D(s,v)-D(t,u)|
            +
            |B(s,v)^{-1}\big(B(t,u)-B(s,v)\big)B(t,u)^{-1}|\,|D(t,u)|
            \\\leq&~
            |B(s,v)^{-1}|\,|D(s,v)-D(t,u)|
            +
            |B(s,v)^{-1}|\,|B(t,u)-B(s,v)|\,|B(t,u)^{-1}|\,|D(t,u)| 
            \\\leq&~
            C \,|D(s,v)-D(t,u)| + C(1+|u|)^\gamma\,|B(t,u)-B(s,v)|.
      \end{align*}
      Using Assumption \ref{ass:time-regularity} and Lemma \ref{lem:prop:A}, we obtain
      \begin{align*}
            &~|D(s,v)-D(t,u)|
            =
            \big|A_s' \hat{V}(s,v)+R \partial_tF \big(s,v+\hat{V}(s,v)\big)-A_t' \hat{V}(t,u) - R \partial_tF \big(t,u+\hat{V}(t,u)\big)\big|
            \\\leq&~
            |A_s' - A_t'||\hat{V}(s,v)| + |A_t'||\hat{V}(s,v)-\hat{V}(t,u)|
            +
            |R|\big|\partial_tF \big(s,v+\hat{V}(s,v)\big)-\partial_tF \big(t,u+\hat{V}(t,u)\big)\big|
            \\\leq&~
            C|s-t|\big(1+|v|\big)+C\big(1+|u|+|v|\big)^{\gamma}
            \big|s-t\big|+C\big|v-u\big|
            \\&~+
            C(1+| u+\hat{V}(t,u) |+| v+\hat{V}(s,v) |)^{\gamma-1} | (v+\hat{V}(s,v)) - (u+\hat{V}(t,u)) |
            \\
            &~+C(1+| u+\hat{V}(t,u) |+| v+\hat{V}(s,v) |)^{\gamma}|s-t|\\
            \leq&~
            C(1+| u |+| v |)^{\gamma-1} | v - u |+C(1+| u |+| v |)^{2\gamma-1}|s-t|,
      \end{align*}
      and
      \begin{align*}
            &~|B(t,u)-B(s,v)|
            =
            \big|A_t+R\partial_xF(t,u+\hat{V}(t,u))
            -A_s - R\partial_xF(s,v+\hat{V}(s,v))\big|
            \\
            \leq&~
            |A_t-A_s|
            +
            |R|\,\big|\partial_xF(t,u+\hat{V}(t,u))
            -\partial_xF(s,v+\hat{V}(s,v))\big|
            \\
            \leq&~
            C|t-s|
            +
            C(1+|u+\hat{V}(t,u)|+|v+\hat{V}(s,v)|)^{\gamma-2}
            |(u+\hat{V}(t,u))-(v+\hat{V}(s,v))|
            \\
            &~+
            C(1+|u+\hat{V}(t,u)|+|v+\hat{V}(s,v)|)^\gamma
            |t-s|
            \\
            \leq&~
            C|t-s|
            +
            C(1+|u|+|v|)^{\gamma-2}
            \big(|u-v| + (1+|u|+|v|)^\gamma |t-s|\big)
            +
            C(1+|u|+|v|)^\gamma |t-s|
            \\
            \leq&~
            C(1+|u|+|v|)^\gamma |u-v|
            +
            C(1+|u|+|v|)^{2\gamma-1}|t-s|.
      \end{align*}
      Finally, we obtain
      \begin{align*}
            \big|\partial_t\hat{V}(t,u)-\partial_t\hat{V}(s,v)\big|
            \leq&~
            C(1+|u|+|v|)^{\gamma-1}|u-v|
            +
            C(1+|u|+|v|)^{2\gamma-1}|t-s|
            \\
            &~+
            C(1+|u|)^\gamma
            \big(
            C(1+|u|+|v|)^\gamma |u-v|
            +
            C(1+|u|+|v|)^{2\gamma-1}|t-s|
            \big)
            \\
            \leq&~
            C(1+|u|+|v|)^{3\gamma-1}
            \big(|u-v|+|t-s|\big).
      \end{align*}
      This proves \eqref{lem-v-rs3} and completes the proof.
\end{proof}

The following lemma follows from \cite[Lemma 2.4]{chen2025strong}. The proof is omitted.

\begin{lemma}\label{lem-v-3}
Suppose that Assumptions \ref{asm:At}, \ref{ass:index1}, \ref{asm:FG}, and \ref{ass:time-regularity} hold. Then, for any $x,y \in \mathbb{R}^{d}$ and $0 \leq t \leq T$,
      \begin{equation}\label{lem:fg:rs111}
            2\big\langle x - y, 
            f(t,x) - f(t,y) \big\rangle 
           +\big(p_{1}-1\big)
            \big|g(t,x) - g(t,y)\big|^{2}
            \leq
            2L_{1}\big(1+\hat{L}^2\big)
            \big|x-y\big|^{2}.
      \end{equation}
Moreover, for any $2 \leq p < p_{1}$, there exists a constant $C > 0$ such that
      \begin{equation}\label{lem:fg:rs4}
            2\big\langle x, f(t,x) \big\rangle 
           +\big(p-1\big)\big|g(t,x)\big|^{2}
            \leq 
            C\big(1+\big|x\big|^{2}\big),
            \quad x \in \R^{d}, t \in [0,T].
      \end{equation}
\end{lemma}

The following moment estimates for the solution processes $\{U_t\}_{t\in[0,T]}$ and $\{X_t\}_{t\in[0,T]}$ follow from \cite[Lemma 2.5]{chen2025strong}. The proof is omitted.

\begin{lemma}\label{lem:bound}
Suppose that Assumptions \ref{asm:At}, \ref{ass:index1}, \ref{asm:FG}, and \ref{ass:time-regularity} hold. Then for any $2 \leq p < p_{1}$, there exists a constant $C > 0$ such that
      \begin{align*}
            \sup_{0 \leq t \leq T} 
            \mathbb{E}\big[|U_{t}|^{p}\big]
            +
            \sup_{0 \leq t \leq T} 
            \mathbb{E}\big[|X_{t}|^{p}\big] \leq C.
      \end{align*}
\end{lemma}

The following H\"{o}lder continuity estimate for $\{U_t\}_{t\in[0,T]}$ and $\{X_t\}_{t\in[0,T]}$ follows from \cite[Lemma 2.6]{chen2025strong}. The proof is omitted.

\begin{lemma}\label{lem:UX}
Suppose that Assumptions \ref{asm:At}, \ref{ass:index1}, \ref{asm:FG}, and \ref{ass:time-regularity} hold. Then, for any $2\leq \nu<p_1/\gamma$, there exists a constant $C>0$ such that, for all $s,t\in[0,T]$,
      \begin{align*}
            \mathbb{E}\big[|U_t - U_s|^{\nu}\big]
            +
            \mathbb{E}\big[\big|X_t- X_s\big|^{\nu}\big]
            \leq 
            C\Big(|t-s|^{\nu}
            +
            |t-s|^{\frac{\nu}{2}}\Big).
      \end{align*}
\end{lemma}

As a direct consequence of Lemmas \ref{lem:bound} and \ref{lem:UX}, we obtain the following weighted increment estimate.

\begin{corollary}\label{cor:mixed-moment}
Suppose that Assumptions \ref{asm:At}, \ref{ass:index1}, \ref{asm:FG}, and \ref{ass:time-regularity} hold. Let $a,b>0$ satisfy $2\gamma\leq a+b\gamma < p_1 $. Then there exists a constant $C>0$ such that, for all $s,t\in[0,T]$,
      \begin{equation*}
      \mathbb E\left[ (1+|X_t|+|X_s|)^a |X_t-X_s|^b \right] \leq C|t-s|^{\frac b2}.
      \end{equation*}
\end{corollary}

\begin{proof}
Set $\nu=(a+b\gamma)/\gamma$. Then $2\leq\nu<p_1/\gamma$. By H\"older's inequality, Lemma \ref{lem:bound}, and Lemma \ref{lem:UX},
      \begin{align*}
            \mathbb E\left[
            (1+|X_t|+|X_s|)^a|X_t-X_s|^b
            \right]
            &\leq
            \left(
            \mathbb E\left[(1+|X_t|+|X_s|)^{a+b\gamma}\right]
            \right)^{\frac{a}{a+b\gamma}}
            \left(\mathbb E\left[|X_t-X_s|^\nu\right]\right)^{\frac{b\gamma}{a+b\gamma}}
            \\
            &\leq
            C\left(|t-s|^\nu+|t-s|^{\nu/2}\right)^{b/\nu}
            \\
            &\leq C|t-s|^{b/2}.
      \end{align*}
      The last inequality follows from $0\leq|t-s|\leq T$.
\end{proof}

The following lemma gives the mean square time-regularity of the SDAE coefficients and the reduced coefficients along the exact solutions of \eqref{eq:SDAE} and \eqref{eq:inherent_U}.

\begin{lemma}\label{lem:reduced-ms-regularity}
      Suppose that Assumptions \ref{asm:At}, \ref{ass:index1}, \ref{asm:FG}, and \ref{ass:time-regularity} hold. Then there exists a constant $C>0$ such that for all $s,t\in[0,T]$,
      \begin{align}
            &~\mathbb{E}\left[|F(t,X_t)-F(s,X_s)|^2\right]
            +
            \mathbb{E}\left[|G(t,X_t)-G(s,X_s)|^2\right]
            \leq C|t-s|,
            \notag\\
            &~\mathbb{E}\left[|f(t,U_t)-f(s,U_s)|^2\right]
            +
            \mathbb{E}\left[|g(t,U_t)-g(s,U_s)|^2\right]
            \leq C|t-s|.
            \label{eq:red-ms-fg}
      \end{align}
\end{lemma}

\begin{proof}
      By \eqref{asm:FG-absFa}, \eqref{asm:FG-absG} and $\gamma \geq 1$,
      \begin{align*}
            &~|F(t, X_t) - F(s, X_s)|^2+|G(t, X_t) - G(s, X_s)|^2 \\
            \leq&~ 
            C \left( (1+|X_t|+|X_s|)^{\gamma - 1} |X_t - X_s|+(1+|X_t|+|X_s|)^{\gamma} |t - s| \right)^2 \\
            &~+ 
            C \left( (1+|X_t|+|X_s|)^{\gamma - 1} |X_t - X_s|^2+(1+|X_t|+|X_s|)^{\gamma+1} |t - s|^2 \right) \\
            \leq&~ C \left( (1+|X_t|+|X_s|)^{2\gamma - 2} |X_t - X_s|^2+(1+|X_t|+|X_s|)^{2\gamma } |t - s|^2 \right) \\
            \leq&~ C \left( (1+|X_t|+|X_s|)^{2\gamma } |X_t - X_s|^2+(1+|X_t|+|X_s|)^{2\gamma } |t - s|^2 \right) .
      \end{align*}
      Since $2\gamma + 2\gamma = 4\gamma<p_1$, Corollary \ref{cor:mixed-moment}, with $a=2\gamma$ and $b=2$, and Lemma \ref{lem:bound} give
      \begin{align}\label{FG-absDD}
            &~\mathbb{E} \left[ |F(t, X_t) - F(s, X_s)|^2 \right]
           +\mathbb{E} \left[ |G(t, X_t) - G(s, X_s)|^2 \right]
            \notag
            \\\leq&~
            C\mathbb E\left[
            (1+|X_t|+|X_s|)^{2\gamma}|X_t-X_s|^2
            \right]
            +
            C|t-s|^2\mathbb E\left[(1+|X_t|+|X_s|)^{2\gamma}\right]
            \leq C|t-s|.
      \end{align}
      By \eqref{FG-absDD} and Lemma \ref{lem:prop:A} we have
      \begin{align*}
            &~\mathbb{E} \left[ |f(t, U_t) - f(s, U_s)|^2 \right]
            =
            \mathbb{E} \left[ |A_t ^-F(t, X_t) - A_s ^-F(s, X_s)|^2 \right]
            \notag \\
            \leq&~
            2\mathbb{E} \left[ |A_t ^-F(t, X_t) - A_s ^-F(t, X_t)|^2 \right]+
            2\mathbb{E} \left[ |A_s ^-F(t, X_t) - A_s ^-F(s, X_s)|^2 \right]
            \notag \\
            \leq&~
            2 |A_t ^- - A_s ^-|^2\mathbb{E} \left[|F(t, X_t)|^2 \right]+
            2|A_s ^-|^2\mathbb{E} \left[ |F(t, X_t) - F(s, X_s)|^2 \right] \notag
            \\
            \leq&~ C|t-s|.
      \end{align*}
      Similarly, we obtain $\mathbb{E} \left[ |g(t, U_t) - g(s, U_s)|^2 \right] \leq C|t-s|$. Therefore \eqref{eq:red-ms-fg} follows.
\end{proof}

\subsection{Differential properties of the reconstruction map}

Fix $t\in[0,T]$ and $u\in\R^d$, and set $x:=\Psi_t(u)\in\mathcal M_t$. Differentiating the reconstruction equation
\begin{equation*}
      A_t\hat V(t,u)+RF\bigl(t,u+\hat V(t,u)\bigr)=0
\end{equation*}
with respect to $u$ gives $J(t,x)\partial_u\hat V(t,u) = -R\partial_xF(t,x)$. It follows from $J(t,x)=A_t+R\partial_xF(t,x)$ that
\begin{align}\label{eq:DuV-DuPsi}
      \partial_u\hat V(t,u) = J(t,x)^{-1}A_t-I_d,
      \quad
      \partial_u\Psi_t(u) = J(t,x)^{-1}A_t.
\end{align}
The formulas in \eqref{eq:DuV-DuPsi} will be used both in the derivative estimates below and in the identification of the Milstein coefficient in Section \ref{sec:orderAt}.

For the Milstein analysis, we need the following polynomial growth estimates for the reduced coefficients and their first-order derivatives with respect to the space and time variables.

\begin{lemma}\label{lem:reduced-growth-estimates}
Suppose that Assumptions \ref{asm:At}, \ref{ass:index1}, \ref{asm:FG}, and \ref{ass:time-regularity} hold. Then there exists a constant $C>0$ such that, for all $t\in[0,T]$, $u\in\R^d$, and $j=1,\cdots,m$, with $x:=\Psi_t(u)$, one has
      \begin{alignat}{2}
            \left|\partial_ug_j(t,u)\right|^2
            \leq&~ C(1+|x|)^{\gamma-1}, \quad&
            \left|\partial_tg_j(t,u)\right|^2
            \leq&~ C(1+|x|)^{3\gamma-1},
            \label{eq:red-growth-gu}\\
            \left|\partial_uf(t,u)\right|^2
            \leq&~ C(1+|x|)^{2\gamma}, \quad&
            \left|\partial_tf(t,u)\right|^2
            \leq&~ C(1+|x|)^{4\gamma},
            \label{eq:red-growth-fu}\\
            |f(t,u)|^2
            \leq&~ C(1+|x|)^{2\gamma}, \quad&
            |g(t,u)|^2
            \leq&~ C(1+|x|)^{\gamma+1}.
            \label{eq:red-growth-fg}
      \end{alignat}
\end{lemma}

\begin{proof}
      Note that $\Psi_t(u)\in\mathcal M_t$, then $Px=P \Psi_t(u)=Pu$.
      Moreover, \eqref{eq:DuV-DuPsi}, Assumption \ref{ass:index1}, and Lemma \ref{lem:prop:A} imply $\partial_u\Psi_t(u)=J(t,x)^{-1}A_t$ and $|\partial_u\Psi_t(u)|\leq C$. Since $\Psi_t(u)=\Psi_t(Pu)$, differentiation with respect to $t$ and Lemma \ref{lem-v} give
      \begin{equation*}
            \partial_t\Psi_t(u)=\partial_t\Psi_t(Pu),
            \quad
            |\partial_t\Psi_t(u)| \leq C(1+|Pu|)^\gamma \leq C(1+|x|)^\gamma .
      \end{equation*}
      By Assumption \ref{ass:time-regularity} and Remark \ref{rm:FG}, using the uniform boundedness of the coefficients and their derivatives at the origin on the compact interval $[0,T]$, we obtain
      \begin{gather}
            |F(t,x)|^2+|\partial_xF(t,x)|^2+|\partial_tF(t,x)|^2
            \leq C(1+|x|)^{2\gamma},\label{eq:red-growth-F-basic}
            \\
            |G_j(t,x)|^2+|\partial_tG_j(t,x)|^2
            \leq C(1+|x|)^{\gamma+1},
            \quad
            |\partial_xG_j(t,x)|^2
            \leq C(1+|x|)^{\gamma-1}.
            \label{eq:red-growth-DxG-bound}
      \end{gather}
      For the spatial derivatives, the chain rule gives
      \begin{equation}\label{eq:reduced-spatial-chain-rule}
            \partial_ug_j(t,u)
            =A_t^-\partial_xG_j(t,x)\partial_u\Psi_t(u),
            \quad
            \partial_uf(t,u)
            =A_t^-\partial_xF(t,x)\partial_u\Psi_t(u).
      \end{equation}
      Combining \eqref{eq:reduced-spatial-chain-rule} with Lemma \ref{lem:prop:A}, \eqref{eq:DuV-DuPsi}, \eqref{eq:red-growth-F-basic} and \eqref{eq:red-growth-DxG-bound} gives the spatial-derivative bounds in \eqref{eq:red-growth-gu} and \eqref{eq:red-growth-fu}.

      For the time derivatives, again by the chain rule,
      \begin{gather}
            \partial_tg_j(t,u)
            =
            (A_t^-)'G_j(t,x)
            +
            A_t^-\left(
            \partial_tG_j(t,x)+\partial_xG_j(t,x)\partial_t\Psi_t(u)
            \right),\label{eq:gt-time-derivative}
            \\
            \partial_tf(t,u)
            =
            (A_t^-)'F(t,x)
            +
            A_t^-\left(
            \partial_tF(t,x)+\partial_xF(t,x)\partial_t\Psi_t(u)
            \right). \label{eq:ft-time-derivative}
      \end{gather}
      Hence, using Lemmas \ref{lem:prop:A} and \ref{lem-v}, together with \eqref{eq:red-growth-F-basic} and \eqref{eq:red-growth-DxG-bound}, we have
      \begin{align*}
            \left|\partial_tg_j(t,u)\right|^2
            \leq&~
            C(1+|x|)^{\gamma+1}
            +
            C(1+|x|)^{\gamma-1}(1+|x|)^{2\gamma}
            \leq C(1+|x|)^{3\gamma-1},
            \\
            \left|\partial_tf(t,u)\right|^2
            \leq&~
            C(1+|x|)^{2\gamma}
            +
            C(1+|x|)^{2\gamma}(1+|x|)^{2\gamma}
            \leq C(1+|x|)^{4\gamma}.
      \end{align*}
      Finally, the definitions of $f$ and $g$, Lemma \ref{lem:prop:A}, and \eqref{eq:red-growth-F-basic}--\eqref{eq:red-growth-DxG-bound} imply
      \begin{equation*}
            |f(t,u)|^2 = |A_t^-F(t,x)|^2 \leq C(1+|x|)^{2\gamma}, 
            \quad 
            |g(t,u)|^2 =|A_t^-G(t,x)|^2 \leq C(1+|x|)^{\gamma+1}.
      \end{equation*}
      This proves \eqref{eq:red-growth-gu}--\eqref{eq:red-growth-fg}.
\end{proof}

The following lemma provides polynomially weighted difference estimates for the first-order derivatives of the reduced coefficients. These estimates will be used to control the Taylor remainder terms in the local truncation error analysis.

\begin{lemma}\label{lem:reduced-derivative-difference-estimates}
Suppose that Assumptions \ref{asm:At}, \ref{ass:index1}, \ref{asm:FG}, and \ref{ass:time-regularity} hold. Then there exists a constant $C>0$ such that, for any $s,t\in[0,T]$, $u,v\in\R^d$, and $j=1,\cdots,m$, with $x:=\Psi_t(u)$ and $y:=\Psi_s(v)$, one has
      \begin{align}
            \left|
            \partial_ug_j(t,u)
            -
            \partial_ug_j(s,v)
            \right|^2
            \leq&~
            C(1+|x|+|y|)^{3\gamma-1}
            \big(|x-y|^2+|t-s|^2\big),
            \label{eq:red-diff-gu}
            \\
            \left|
            \partial_tg_j(t,u)
            -
            \partial_tg_j(s,v)
            \right|^2
            \leq&~
            C(1+|x|+|y|)^{7\gamma-3}
            \big(|x-y|^2+|t-s|^2\big),
            \label{eq:red-diff-gt}
            \\
            \left|
            \partial_uf(t,u)
            -
            \partial_uf(s,v)
            \right|^2
            \leq&~
            C(1+|x|+|y|)^{4\gamma}
            \big(|x-y|^2+|t-s|^2\big),
            \label{eq:red-diff-fu}
            \\
            \left|
            \partial_tf(t,u)
            -
            \partial_tf(s,v)
            \right|^2
            \leq&~
            C(1+|x|+|y|)^{8\gamma-2}
            \big(|x-y|^2+|t-s|^2\big).
            \label{eq:red-diff-ft}
      \end{align}
\end{lemma}

\begin{proof}
      Note that $\Psi_t(u)\in\mathcal M_t$ and $\Psi_s(v)\in\mathcal M_t$, then 
      $$Px=P \Psi_t(u)=Pu, \quad Py=P \Psi_s(v)=Pv.$$
      By \eqref{eq:DuV-DuPsi}, $\partial_u\Psi_t(u)=J(t,x)^{-1}A_t$ and $\partial_u\Psi_s(v)=J(s,y)^{-1}A_s$. Using Assumptions \ref{ass:index1} and \ref{asm:FG-absIa}, we obtain
      \begin{align*}
            |J(t,x)^{-1}-J(s,y)^{-1}|
            \leq&~
            |J(t,x)^{-1}| \, \big| J(s,y)-J(t,x)\big| \, |J(s,y)^{-1}|
            \notag\\
            \leq&~
            L_J^2 |A_t+R\partial_x F(t,x) -A_s-R\partial_x F(s,y) |
            \notag\\
            \leq&~
            C \big(|t-s|+ (1+|x|+| y|)^{\gamma-2} | x 
            - y |+(1+|x|+|y|)^{\gamma}|t-s| \big)
            \notag\\
            \leq&~
            C(1+|x|+|y|)^\gamma
            \big(|x-y|+|t-s|\big).
      \end{align*}
      This together with Lemma \ref{lem:prop:A} yields
      \begin{align}\label{eq:red-diff-B}
            |J(t,x)^{-1}A_t-J(s,y)^{-1}A_s|
            \leq&~
            |J(t,x)^{-1}|\,|A_t-A_s|
            +
            |J(t,x)^{-1}|\,
            |J(t,x)-J(s,y)|\,
            |J(s,y)^{-1}|\,|A_s|
            \notag\\
            \leq&~
            C(1+|x|+|y|)^\gamma
            \big(|x-y|+|t-s|\big).
      \end{align}
      Also, differentiating the first relation in \eqref{eq:projection-invariant-extension} with respect to time and applying \eqref{lem-v-rs3} at $Pu$ and $Pv$ give
      \begin{align}\label{eq:red-diff-DtPsi}
            |\partial_t\Psi_t(u)-\partial_t\Psi_s(v)|
            =&
            |\partial_t\Psi_t(Pu)-\partial_t\Psi_s(Pv)|
            \notag\\
            \leq&~
            C(1+|Px|+|Py|)^{3\gamma-1}
            \big(|Px-Py|+|t-s|\big)
            \notag\\
            \leq&~
            C(1+|x|+|y|)^{3\gamma-1}
            \big(|x-y|+|t-s|\big).
      \end{align}
      We first prove \eqref{eq:red-diff-gu}. By \eqref{eq:reduced-spatial-chain-rule} and \eqref{eq:DuV-DuPsi},
      \begin{align*}
            &~\left|\partial_ug_j(t,u)-\partial_ug_j(s,v)\right|^2
            \\\leq&~
            C|A_t^--A_s^-|^2|\partial_xG_j(t,x)|^2|J(t,x)^{-1}A_t|^2
            +
            C|A_s^-|^2|\partial_xG_j(t,x)-\partial_xG_j(s,y)|^2|J(t,x)^{-1}A_t|^2
            \\
            &~+
            C|A_s^-|^2|\partial_xG_j(s,y)|^2|J(t,x)^{-1}A_t-J(s,y)^{-1}A_s|^2 .
      \end{align*}
      This together with \eqref{asm:FG-absKa}, \eqref{asm:dG}, \eqref{eq:red-diff-B}, and Lemma \ref{lem:prop:A} yields \eqref{eq:red-diff-gu}. 
      For $\partial_uf$, the second formula in \eqref{eq:reduced-spatial-chain-rule} and the same decomposition, together with \eqref{asm:FG-absIa}, \eqref{asm:dF}, \eqref{eq:red-diff-B}, and Lemma \ref{lem:prop:A} yield \eqref{eq:red-diff-fu}. The exponent $4\gamma$ results from combining the polynomial growth of $\partial_xF$ with the polynomial difference bound for the reconstruction derivative $J(t,x)^{-1}A_t$.

      It remains to estimate the time derivatives. For $g_j$, applying \eqref{eq:gt-time-derivative} at $(t,u)$ and $(s,v)$ gives
      \begin{align*}
            \left|
            \partial_tg_j(t,u)
            -
            \partial_tg_j(s,v)
            \right|^2
            \leq&~
            C\left|
            (A_t^-)'G_j(t,x)-(A_s^-)'G_j(s,y)
            \right|^2
            +
            C\left|
            A_t^-\partial_tG_j(t,x)-A_s^-\partial_tG_j(s,y)
            \right|^2
            \\
            &~+
            C\left|
            A_t^-\partial_xG_j(t,x)\partial_t\Psi_t(u)
            -
            A_s^-\partial_xG_j(s,y)\partial_t\Psi_s(v)
            \right|^2 .
      \end{align*}
      By Lemma \ref{lem:prop:A}, \eqref{asm:FG-absG} and \eqref{asm:bound_FG},
      \begin{align*}
            \left|
            (A_t^-)'G_j(t,x)-(A_s^-)'G_j(s,y)
            \right|^2
            \leq&~
            C|(A_t^-)'-(A_s^-)'|^2|G_j(t,x)|^2
            +C|(A_s^-)'|^2|G_j(t,x)-G_j(s,y)|^2
            \\\leq&~
            C(1+|x|+|y|)^{2\gamma}
            \big(|x-y|^2+|t-s|^2\big).
      \end{align*}
      Similarly, Lemma \ref{lem:prop:A}, \eqref{asm:FG-absMa} and \eqref{asm:dG} imply
      \begin{align*}
            \left|
            A_t^-\partial_tG_j(t,x)-A_s^-\partial_tG_j(s,y)
            \right|^2
            \leq&~
            C|A_t^--A_s^-|^2|\partial_tG_j(t,x)|^2
            +C|A_s^-|^2|\partial_tG_j(t,x)-\partial_tG_j(s,y)|^2
            \\\leq&~
            C(1+|x|+|y|)^{2\gamma}
            \big(|x-y|^2+|t-s|^2\big).
      \end{align*}
      For the last difference, we split it according to the three factors. Using \eqref{asm:FG-absKa}, \eqref{asm:dG}, \eqref{eq:red-diff-DtPsi}, Lemma \ref{lem:prop:A}, and the estimate $|\partial_t\Psi_t(u)|\leq C(1+|x|)^\gamma$ as in Lemma \ref{lem:reduced-growth-estimates}, we derive
      \begin{align*}
            &~\left|A_t^-\partial_xG_j(t,x)\partial_t\Psi_t(u)
            - A_s^-\partial_xG_j(s,y)\partial_t\Psi_s(v)\right|^2
            \\\leq&~
            C|A_t^--A_s^-|^2
            |\partial_xG_j(t,x)|^2
            |\partial_t\Psi_t(u)|^2
            +
            C|A_s^-|^2
            |\partial_xG_j(t,x)-\partial_xG_j(s,y)|^2
            |\partial_t\Psi_t(u)|^2
            \\&~+
            C|A_s^-|^2
            |\partial_xG_j(s,y)|^2
            |\partial_t\Psi_t(u)-\partial_t\Psi_s(v)|^2
            \\\leq&~
            C(1+|x|+|y|)^{3\gamma-1}|t-s|^2
            +
            C(1+|x|+|y|)^{3\gamma-3}|x-y|^2
            +C(1+|x|+|y|)^{3\gamma+1}|t-s|^2
            \\&~+
            C(1+|x|+|y|)^{7\gamma-3}
            \big(|x-y|^2+|t-s|^2\big)
            \\\leq&~
            C(1+|x|+|y|)^{7\gamma-3}
            \big(|x-y|^2+|t-s|^2\big).
      \end{align*}
      Combining the above three estimates gives \eqref{eq:red-diff-gt}.

      It remains to prove \eqref{eq:red-diff-ft}. By \eqref{eq:ft-time-derivative},
      \begin{align*}
            \partial_tf(t,u) - \partial_tf(s,v)
            =&~
            \big((A_t^-)'-(A_s^-)'\big)F(t,x)
            +
            (A_s^-)'\big(F(t,x)-F(s,y)\big)
            \\
            &~+
            (A_t^- - A_s^-)\partial_tF(t,x)
            +
            A_s^-\big(\partial_tF(t,x)-\partial_tF(s,y)\big)
            \\
            &~+
            (A_t^- - A_s^-)\partial_xF(t,x)\partial_t\Psi_t(u)
            +
            A_s^-\big(\partial_xF(t,x)-\partial_xF(s,y)\big)\partial_t\Psi_t(u)
            \\
            &~+
            A_s^-\partial_xF(s,y)
            \big(\partial_t\Psi_t(u)-\partial_t\Psi_s(v)\big).
      \end{align*}
      Hence, by Lemma \ref{lem:prop:A}, Lemma \ref{lem-v}, and Assumption \ref{ass:time-regularity},
      \begin{align*}
            &~\left|\partial_tf(t,u) - \partial_tf(s,v)\right|^2
            \\\leq&~
            C|t-s|^2|F(t,x)|^2
            +
            C|F(t,x)-F(s,y)|^2
            +
            C|t-s|^2|\partial_tF(t,x)|^2
            +
            C|\partial_tF(t,x)-\partial_tF(s,y)|^2
            \\&~+
            C|t-s|^2
            |\partial_xF(t,x)|^2
            |\partial_t\Psi_t(u)|^2
            +
            C|\partial_xF(t,x)-\partial_xF(s,y)|^2
            |\partial_t\Psi_t(u)|^2
            \\
            &~+
            C|\partial_xF(s,y)|^2
            |\partial_t\Psi_t(u)-\partial_t\Psi_s(v)|^2
            \\\leq&~
            C(1+|x|+|y|)^{8\gamma-2}
            \big(|x-y|^2+|t-s|^2\big).
      \end{align*}
      Therefore \eqref{eq:red-diff-ft} follows.
\end{proof}

\section{Stochastic theta Milstein method and its mean square convergence}\label{sec:orderAt}
The stochastic theta Milstein method analyzed below is defined by \eqref{eq:ABEM}. Under Assumptions \ref{asm:At}, \ref{ass:index1}, \ref{asm:FG}, and \ref{ass:time-regularity}, we establish its well-posedness, constraint preserving property, and mean square convergence rate. Proposition \ref{prop:exact-solution-constraint} ensures that the exact solution satisfies $X_t\in\mathcal M_t$ for all $t\in[0,T]$. Moreover, for any $t\in[0,T]$ and $x\in\mathcal M_t$,
\begin{gather}\label{def:Qx=v}
      x = Px+\hat{V}(t,Px),
      \quad Qx = \hat{V}(t,Px),
      \\\label{def:Ff}
      A_t^- F(t,x) = f(t,Px),
      \quad
      A_t^- G(t,x) = g(t,Px).
\end{gather}
Throughout this section, let $\theta\in[\frac12,1]$ and define
\begin{equation}\label{eq:stepsize-threshold}
      \bar\Delta
      :=
      \min\left\{
      \Delta_0,
      \frac{1}{4\theta L_1(1+\hat L^2)}
      \right\}.
\end{equation}
We consider $\Delta=T/K$ satisfying $0<\Delta\leq\bar\Delta$.

\subsection{SDAE-level Milstein coefficient and reduced formulation}

Fix $t\in[0,T]$ and $u\in\R^d$, and set $x:=\Psi_t(u)\in\mathcal M_t$. By \eqref{eq:DuV-DuPsi}, the chain rule, $A_tA_t^-=I-R$, and $RG_{j_1}(t,x)=0$, we obtain
\begin{align}\label{eq:reduced-milstein-coefficient}
      \partial_ug_{j_2}(t,u)g_{j_1}(t,u)
      =&~
      A_t^-\partial_xG_{j_2}(t,x)
      \partial_u\Psi_t(u)A_t^-G_{j_1}(t,x)
      \notag\\
      =&~
      A_t^-\partial_xG_{j_2}(t,x)
      J(t,x)^{-1}(I-R)G_{j_1}(t,x)
      \notag\\
      =&~
      A_t^-\widetilde{\mathcal L}^{j_1}G_{j_2}(t,x).
\end{align}
Thus, after projection by $A_t^-$, the SDAE-level correction in \eqref{eq:ABEM} coincides exactly with the standard Milstein coefficient of the reduced SDE. Moreover, differentiating the $j_2$th column of \eqref{eq:noise-compatibility} with respect to $x$ and using that $R$ is independent of $x$ gives
\begin{equation}\label{eq:R-tildeL-zero}
      R\widetilde{\mathcal L}^{j_1}G_{j_2}(t,x)
      =
      R\partial_xG_{j_2}(t,x)
      J(t,x)^{-1}G_{j_1}(t,x)
      =0.
\end{equation}
Hence the Milstein correction has no algebraic component and is compatible with the constraint.

\subsection{Well-posedness of numerical solutions}
We now prove the existence and uniqueness of numerical solutions $\{x_k\}_{k=0,1,\cdots,K}$ satisfying $x_k\in\mathcal M_{t_k}$.

\begin{lemma}\label{lem:def}
Suppose that Assumptions \ref{asm:At}, \ref{ass:index1}, \ref{asm:FG}, and \ref{ass:time-regularity} hold. Then, for every $0<\Delta\leq\bar\Delta$, the stochastic theta Milstein method \eqref{eq:ABEM} is well defined and $x_k\in\mathcal M_{t_k}$ for $k=0,1,\cdots,K$.
\end{lemma}

\begin{proof}
We proceed by induction on $k$. Assumption \ref{ass:index1} gives $x_0\in\mathcal M_{t_0}$. Assume that $x_k\in\mathcal M_{t_k}$ and define
      \begin{align*}
            H(x)
            :=&~ 
            A_{t_{k}}x - \theta \big(R+A_{t_k}A_{t_{k+1}}^-\big)
            F(t_{k+1},x)\Delta 
            - \Big(A_{t_k}x_k 
           +
            (1-\theta) F(t_{k},x_{k})\Delta \\&~
           +G(t_k,x_k) \Delta W_k
            +\sum_{j_1,j_2=1}^m \widetilde{\mathcal{L}}^{j_1} G_{j_2}(t_k, x_k)  
            I_{j_1,j_2}^{t_k,t_{k+1}}\Big),
            \quad x \in \R^{d}.
      \end{align*}
      Then \eqref{eq:ABEM} is equivalent to $H(x_{k+1})=0$. Since $R=I-A_{t_k}A_{t_k}^-$, the equation $H(x)=0$ is equivalent to
      \begin{align*}
            \begin{cases}
                   A_{t_{k}}A_{t_{k}}^- H(x) = 0,
                   \\
                   RH(x)= 0.
            \end{cases}
      \end{align*}
      Because $A_{t_k}^-A_{t_k}A_{t_k}^-=A_{t_k}^-$, the first equation above is equivalent to $A_{t_k}^-H(x)=0$. Hence $H(x)=0$ is equivalent to
      \begin{align}\label{lem:def:A^-H}
             \begin{cases}
                   A_{t_{k}}^- H(x)= 0,
                   \\
                   R H(x)= 0.
            \end{cases}
      \end{align}
      We now show that \eqref{lem:def:A^-H} admits a unique solution $x=x_{k+1}$. Define $\Phi_{k+1}(u):=u-\theta\Delta f(t_{k+1},u)$. By \eqref{lem:fg:rs111}, for any $u,v\in\mathbb R^d$,
      \begin{align*}
            \langle u-v,\Phi_{k+1}(u)-\Phi_{k+1}(v)\rangle
            \geq
            \bigl(1-\theta\Delta L_1(1+\hat L^2)\bigr)|u-v|^2
            \geq\frac34|u-v|^2.
      \end{align*}
      Thus $\Phi_{k+1}$ is continuous and uniformly strongly monotone. The uniform monotonicity theorem in \cite[Theorem C.2]{MR1402909} therefore implies that the following equation
      \begin{align*}
            &~x - \theta f(t_{k+1},x) \Delta 
            - \big(P x_k + (1-\theta)A_{t_{k}}^-F(t_{k},x_{k}) \Delta
            \\&~
            + A_{t_{k}}^-G(t_{k},x_{k}) \Delta W_k+\sum_{j_1,j_2=1}^m 
            A_{t_{k}}^-\widetilde{\mathcal{L}}^{j_1} G_{j_2}(t_k, x_k)  
            I_{j_1,j_2}^{t_k,t_{k+1}}\big)
            = 0, \quad x \in \mathbb{R}^{d},
      \end{align*}
      has a unique solution $u_{k+1}\in\mathbb R^d$, namely,
      \begin{align}\label{lem:eq:u}
            &~u_{k+1} - \theta f(t_{k+1},u_{k+1}) \Delta 
            - \big(P x_k + (1-\theta)A_{t_{k}}^-F(t_{k},x_{k}) \Delta \notag
            \\&~+
            A_{t_{k}}^-G(t_{k},x_{k}) \Delta W_k+\sum_{j_1,j_2=1}^m 
            A_{t_{k}}^-\widetilde{\mathcal{L}}^{j_1} G_{j_2}(t_k, x_k)  
            I_{j_1,j_2}^{t_k,t_{k+1}}\big)
            = 0.
      \end{align}
      On the one hand, applying $Pf(t,x)=f(t,x)$ and $PA^-_t=A^-_t$ leads to
      \begin{align*}
            Pu_{k+1}
             =&~
            \theta Pf(t_{k+1},u_{k+1}) \Delta 
           +P^{2} x_k
           +(1-\theta) P A_{t_{k}}^-F(t_{k},x_{k}) \Delta 
            \\&~
           +PA_{t_{k}}^-G(t_{k},x_{k}) \Delta W_k+\sum_{j_1,j_2=1}^m 
            PA_{t_{k}}^-\widetilde{\mathcal{L}}^{j_1} G_{j_2}(t_k, x_k)  
            I_{j_1,j_2}^{t_k,t_{k+1}}
            \\
            =&~
            \theta f(t_{k+1},u_{k+1}) \Delta 
           +P x_k
           +(1-\theta)A_{t_{k}}^-F(t_{k},x_{k}) \Delta 
            \\&~
           +A_{t_{k}}^-G(t_{k},x_{k}) \Delta W_k+\sum_{j_1,j_2=1}^m 
            A_{t_{k}}^-\widetilde{\mathcal{L}}^{j_1} G_{j_2}(t_k, x_k)  
            I_{j_1,j_2}^{t_k,t_{k+1}}
            \\
             =&~
            u_{k+1} \in \operatorname{Im}(P).
      \end{align*}
      On the other hand, there exists a unique $v_{k+1} = \hat{V}(t_{k+1},u_{k+1}) \in \mathbb{R}^{d}$ such that
      \begin{align*}
            A_{t_{k+1}}v_{k+1} = 0,
            \quad 
            RF(t_{k+1},u_{k+1}+v_{k+1})=0.
      \end{align*}
      Consequently, $Pv_{k+1}=A_{t_{k+1}}^-A_{t_{k+1}}v_{k+1}=0$ and $Qv_{k+1}=(I-P)v_{k+1}=v_{k+1}$. Set $x_{k+1}:=u_{k+1}+v_{k+1}$. Then $Px_{k+1}=u_{k+1}$, $Qx_{k+1}=v_{k+1}$, and $x_{k+1}\in\mathcal M_{t_{k+1}}$. It follows from \eqref{lem:eq:u}, $A_t^-R=0$, and $A_{t_{k+1}}^-F(t_{k+1},x_{k+1})=f(t_{k+1},u_{k+1})$ that
      \begin{align*}
            A_{t_{k}}^- H(x_{k+1})
            =&~
            A_{t_{k}}^-A_{t_{k}}x_{k+1} 
            - \theta A_{t_{k}}^-
            \big(R+A_{t_k}A_{t_{k+1}}^-\big)
            F(t_{k+1},x_{k+1}) \Delta 
            -
            \big( A_{t_{k}}^-A_{t_k}x_k
            \\&~+ (1-\theta) A_{t_{k}}^-
            F(t_{k},x_{k}) \Delta  
           +A_{t_{k}}^-G(t_k,x_k) 
            \Delta W_k+\sum_{j_1,j_2=1}^m A_{t_{k}}^-
            \widetilde{\mathcal{L}}^{j_1} G_{j_2}(t_k, x_k)  
            I_{j_1,j_2}^{t_k,t_{k+1}}\big)
            \\
            =&~
            P x_{k+1} 
            - \theta A_{t_{k+1}}^-F(t_{k+1},x_{k+1}) \Delta 
            -\big( P x_k+ (1-\theta) A_{t_{k}}^-
            F(t_{k},x_{k}) \Delta  
            \\&~+ A_{t_{k}}^-G(t_k,x_k) \Delta W_k
            +\sum_{j_1,j_2=1}^m A_{t_{k}}^-
            \widetilde{\mathcal{L}}^{j_1} G_{j_2}(t_k, x_k)  
            I_{j_1,j_2}^{t_k,t_{k+1}}\big)
            \\
            =&~
            u_{k+1} - \theta f(t_{k+1},u_{k+1}) \Delta 
            - \big(P x_k+ (1-\theta)A_{t_{k}}^-
            F(t_{k},x_{k}) \Delta\\&~
           +A_{t_{k}}^-G(t_{k},x_{k}) \Delta W_k+
            \sum_{j_1,j_2=1}^m A_{t_{k}}^-
            \widetilde{\mathcal{L}}^{j_1} G_{j_2}(t_k, x_k)  
            I_{j_1,j_2}^{t_k,t_{k+1}}\big)
            \\
            =&~
            0.
      \end{align*}
      Furthermore, using \eqref{eq:R-tildeL-zero}, $R^2=R$, $RA_{t_k}=0$, and $RG(t,x)=0$, we have
      \begin{align*}
            R H(x_{k+1})
            =&~
            R\bigg(A_{t_{k}}x_{k+1} - 
            \theta \big(R+A_{t_k}A_{t_{k+1}}^-\big)
            F(t_{k+1},x_{k+1})\Delta 
            - \Big(A_{t_k}x_k+
            (1-\theta) F(t_{k},x_{k})\Delta 
            \\&~
           +G(t_k,x_k) \Delta W_k
            +\sum_{j_1,j_2=1}^m 
            \widetilde{\mathcal{L}}^{j_1} G_{j_2}(t_k, x_k)  
            I_{j_1,j_2}^{t_k,t_{k+1}}\Big)\bigg)
            \\
            =&~
            -\theta R F(t_{k+1},x_{k+1}) \Delta 
            -(1-\theta) R F(t_{k},x_{k}) \Delta
            -\sum_{j_1,j_2=1}^m R 
            \widetilde{\mathcal{L}}^{j_1} G_{j_2}(t_k, x_k)  
            I_{j_1,j_2}^{t_k,t_{k+1}}\\
            =&~0.
      \end{align*}
      Conversely, let $x\in\mathbb R^d$ satisfy $H(x)=0$. From $RH(x)=0$, $x_k\in\mathcal M_{t_k}$, $RG(t_k,x_k)=0$, and \eqref{eq:R-tildeL-zero}, we obtain $RF(t_{k+1},x)=0$. Hence $x\in\mathcal M_{t_{k+1}}$ and $x=Px+\hat V(t_{k+1},Px)$. Moreover, $A_{t_k}^-H(x)=0$ shows that $Px$ satisfies \eqref{lem:eq:u}. By uniqueness, $Px=u_{k+1}$ and therefore $x=x_{k+1}$. Thus \eqref{eq:ABEM} has a unique solution.

      Finally, since $\Phi_{k+1}^{-1}$ and $u\mapsto u+\hat V(t_{k+1},u)$ are continuous and the right-hand side of \eqref{lem:eq:u} is $\mathcal F_{t_{k+1}}$-measurable, both $u_{k+1}$ and $x_{k+1}$ are $\mathcal F_{t_{k+1}}$-measurable. Thus the numerical solution is adapted, and the induction is complete.
\end{proof}


\subsection{Local truncation error analysis}
The next lemma reduces the global mean square error to local residual estimates. The existence of the exact solution follows from Assumptions \ref{asm:At}, \ref{ass:index1}, \ref{asm:FG}, and \ref{ass:time-regularity}. The well-posedness and constraint preservation of the numerical solution are supplied by Lemma \ref{lem:def}.

\begin{lemma}\label{lem:error-reduction}
Let the assumptions of Lemma \ref{lem:def} hold. Then, for every $0<\Delta\leq\bar\Delta$, there exists a constant $C>0$, independent of $\Delta$, such that

      \begin{align*}
            \max_{1 \leq k \leq K} \mathbb{E} 
            \left[ |X_{t_k} - x_k|^2 \right] 
            \leq C \left( \sum_{k=1}^K \mathbb{E} 
            \left[ |{\mathcal{R}_k}|^2 \right]+
            \Delta^{-1} \sum_{k=1}^K \mathbb{E} 
            \left[ |\mathbb{E}\left[\mathcal{R}_k | \mathcal{F}_{t_{k-1}}\right]|^2 \right] \right),
      \end{align*}
where
      \begin{align}
            \mathcal{R}_k
            :=&~
            \theta
            \int_{t_{k-1}}^{t_k}
            \left(
                  f(s,U_s)
                  -
                  f(t_k,U_{t_k})
            \right) ds
            +
            (1-\theta)
            \int_{t_{k-1}}^{t_k}
            \left(
                  f(s,U_s)
                  -
                  f(t_{k-1},U_{t_{k-1}})
            \right) ds
            \notag
            \\
            &~+
            \int_{t_{k-1}}^{t_k}
            \left(
                  g(s,U_s)
                  -
                  g(t_{k-1},U_{t_{k-1}})
            \right) dW_s
            -
            \sum_{j_1,j_2=1}^m
            A_{t_{k-1}}^-
            \widetilde{\mathcal{L}}^{j_1}G_{j_2}
            (t_{k-1},X_{t_{k-1}})
            I_{j_1,j_2}^{t_{k-1},t_k}.
            \label{eq:local-residual}
      \end{align}

\end{lemma}

\begin{proof} 
Set $u_k:=Px_k$ for $k=0,1,\cdots,K$. Since $P=A_{t_k}^-A_{t_k}$ and $A_{t_k}^-\big(R+A_{t_k}A_{t_{k+1}}^-\big)=A_{t_{k+1}}^-$, applying $A_{t_k}^-$ to \eqref{eq:ABEM} gives
      \begin{align}\label{eq:projected-stm-step}
            u_{k+1}
            =&~A_{t_k}^- A_{t_k} x_{k+1}\notag
            \\
            =&~
            A_{t_k}^-\Big(\theta \big(R+A_{t_k}A_{t_{k+1}}^-\big)
            F(t_{k+1},x_{k+1})\Delta 
           +A_{t_k}x_k 
           +
            (1-\theta) F(t_{k},x_{k})\Delta \notag
            \\&~
           +G(t_k,x_k) \Delta W_k
            +\sum_{j_1,j_2=1}^m \widetilde{\mathcal{L}}^{j_1} G_{j_2}(t_k, x_k)  
            I_{j_1,j_2}^{t_k,t_{k+1}}\Big) \notag
            \\
            =&~
            \theta A_{t_{k+1}}^-F(t_{k+1},x_{k+1}) \Delta 
            +
            P x_k+ (1-\theta) A_{t_{k}}^-
            F(t_{k},x_{k}) \Delta  \notag
            \\
            &~+
            A_{t_{k}}^-G(t_k,x_k) 
            \Delta W_k
            +\sum_{j_1,j_2=1}^m A_{t_{k}}^-
            \widetilde{\mathcal{L}}^{j_1} G_{j_2}(t_k, x_k)  
            I_{j_1,j_2}^{t_k,t_{k+1}}.
      \end{align}
      By Lemma \ref{lem:def}, $x_k\in\mathcal M_{t_k}$ for $k=0,1,\cdots,K$. Hence \eqref{def:Ff} and \eqref{eq:projected-stm-step} yield
      \begin{align}\label{eq:reduced-stm-step}
            u_{k+1} 
            =&~
            u_{k} +\theta f(t_{k+1},u_{k+1}) \Delta 
            +
            (1-\theta)f(t_{k},u_{k}) \Delta \notag
            \\&~+
            g(t_{k},u_{k}) \Delta W_k+
            \sum_{j_1,j_2=1}^m A_{t_{k}}^-
            \widetilde{\mathcal{L}}^{j_1} G_{j_2}(t_k, x_k)  
            I_{j_1,j_2}^{t_k,t_{k+1}}.
      \end{align}
      Letting $e_{k} :=U_{t_k}-u_{k}=PX_{t_k}-Px_{k}=P\big(X_{t_k}-x_{k}\big)$ and using \eqref{eq:inherent_U} and \eqref{eq:reduced-stm-step} yield
      \begin{align*}
            &~e_k - \theta \left( f(t_k, U_{t_k}) - f(t_k, u_k) \right) \Delta
            \\=&~ e_{k-1}+(1 - \theta) \left( f(t_{k-1}, U_{t_{k-1}}) - f(t_{k-1}, u_{k-1}) \right) \Delta 
            + \left( g(t_{k-1}, U_{t_{k-1}}) - g(t_{k-1}, u_{k-1}) \right) \Delta W_{k-1} 
            \\&~+ 
            \sum_{j_1,j_2=1}^m  \big(A_{t_{k-1}}^-\widetilde{\mathcal{L}}^{j_1} G_{j_2}(t_{k-1}, X_{t_{k-1}})-
            A_{t_{k-1}}^-\widetilde{\mathcal{L}}^{j_1} G_{j_2}(t_{k-1}, x_{k-1}) \big)I_{j_1,j_2}^{t_{k-1},t_k}
            +\mathcal{R}_k.
      \end{align*}
      All coefficient differences in the preceding recursion are $\mathcal F_{t_{k-1}}$-measurable and thus can be taken outside conditional expectations. Moreover, the Brownian increments over $[t_{k-1},t_k]$ are independent of $\mathcal F_{t_{k-1}}$. Hence, for $i,j,\ell,r\in\{1,\cdots,m\}$, the following conditional moment identities hold:
      \begin{gather}\label{eq:conditional-moment-formulas}
            \mathbb E\left[
            \Delta W_{k-1}^{i}\mid\mathcal F_{t_{k-1}}
            \right]=0,
            \quad
            \mathbb E\left[
            \Delta W_{k-1}^{i}\Delta W_{k-1}^{j}
            \mid\mathcal F_{t_{k-1}}
            \right]
            =\delta_{ij}\Delta,
            \quad
            \mathbb E\big[
            I_{ij}^{t_{k-1},t_k}\mid\mathcal F_{t_{k-1}}
            \big]=0 \notag
            \\
            \mathbb E\big[
            \Delta W_{k-1}^{i}I_{j\ell}^{t_{k-1},t_k}
            \mid\mathcal F_{t_{k-1}}
            \big]=0,
            \quad
            \mathbb E\big[
            I_{ij}^{t_{k-1},t_k}I_{\ell r}^{t_{k-1},t_k}
            \mid\mathcal F_{t_{k-1}}
            \big]=
            \frac{\Delta^2}{2}\delta_{i\ell}\delta_{jr}.
      \end{gather}
      Here $\delta_{ij}$ denotes the Kronecker delta. Recalling that
      $I_{ij}^{t_{k-1},t_k} = \int_{t_{k-1}}^{t_k} \bigl(W_s^i-W_{t_{k-1}}^i\bigr)\,\mathrm dW_s^j$, the identities in \eqref{eq:conditional-moment-formulas} follow from the martingale property of It\^o integrals and the conditional cross-It\^{o} isometry. Using these identities, we obtain
      \begin{align*}
            &\mathbb E\bigg[\bigg\langle
            e_{k-1}+(1-\theta)
            \big(f(t_{k-1},U_{t_{k-1}})-f(t_{k-1},u_{k-1})\big)\Delta,
            \\
            &\hspace{45mm}
            \big(g(t_{k-1},U_{t_{k-1}})-g(t_{k-1},u_{k-1})\big)
            \Delta W_{k-1}\bigg\rangle\bigg]=0,
            \\
            &\mathbb E\bigg[\bigg\langle
            e_{k-1}+(1-\theta)
            \big(f(t_{k-1},U_{t_{k-1}})-f(t_{k-1},u_{k-1})\big)\Delta,
            \\
            &\hspace{15mm}
            \sum_{j_1,j_2=1}^m
            \big(A_{t_{k-1}}^-\widetilde{\mathcal L}^{j_1}G_{j_2}
            (t_{k-1},X_{t_{k-1}})
            -A_{t_{k-1}}^-\widetilde{\mathcal L}^{j_1}G_{j_2}
            (t_{k-1},x_{k-1})\big)
            I_{j_1,j_2}^{t_{k-1},t_k}
            \bigg\rangle\bigg]=0,
            \\
            &\mathbb E\bigg[\bigg\langle
            \big(g(t_{k-1},U_{t_{k-1}})-g(t_{k-1},u_{k-1})\big)
            \Delta W_{k-1},
            \\
            &\hspace{15mm}
            \sum_{j_1,j_2=1}^m
            \big(A_{t_{k-1}}^-\widetilde{\mathcal L}^{j_1}G_{j_2}
            (t_{k-1},X_{t_{k-1}})
            -A_{t_{k-1}}^-\widetilde{\mathcal L}^{j_1}G_{j_2}
            (t_{k-1},x_{k-1})\big)
            I_{j_1,j_2}^{t_{k-1},t_k}
            \bigg\rangle\bigg]=0.
      \end{align*}
      Moreover,
      \begin{align*}
            &~\mathbb E\bigg[
            \bigg|
            \sum_{j_1,j_2=1}^m
            \big(A_{t_{k-1}}^-\widetilde{\mathcal L}^{j_1}G_{j_2}
            (t_{k-1},X_{t_{k-1}})
            -A_{t_{k-1}}^-\widetilde{\mathcal L}^{j_1}G_{j_2}
            (t_{k-1},x_{k-1})\big)
            I_{j_1,j_2}^{t_{k-1},t_k}
            \bigg|^2\bigg]
            \\=&~
            \frac{\Delta^2}{2}
            \sum_{j_1,j_2=1}^m
            \mathbb E\bigg[
            \big|A_{t_{k-1}}^-\widetilde{\mathcal L}^{j_1}G_{j_2}
            (t_{k-1},X_{t_{k-1}})
            -A_{t_{k-1}}^-\widetilde{\mathcal L}^{j_1}G_{j_2}
            (t_{k-1},x_{k-1})\big|^2\bigg].
      \end{align*}
      The mixed terms involving $\mathcal R_k$ are retained and estimated below. Consequently,
      \begin{align}\label{FG-absG}
            &~\mathbb{E} \left[ |e_k - \theta (f(t_k, U_{t_k}) - f(t_k, u_k)) \Delta|^2 \right] \notag \\
            =&~ \mathbb{E} \left[ |e_{k-1}+(1 - \theta) (f(t_{k-1}, U_{t_{k-1}}) - f(t_{k-1}, u_{k-1})) \Delta|^2 \right] \notag\\
            &~+ \Delta \mathbb{E} \left[ |g(t_{k-1}, U_{t_{k-1}}) - g(t_{k-1}, u_{k-1})|^2 \right]+\mathbb{E} \left[ |\mathcal{R}_k|^2 \right] \notag\\
            &~+\frac{\Delta^2}{2} \sum_{j_1, j_2=1}^m \mathbb{E} [ |A_{t_{k-1}}^-\widetilde{\mathcal{L}}^{j_1} G_{j_2}(t_{k-1}, X_{t_{k-1}})-A_{t_{k-1}}^-\widetilde{\mathcal{L}}^{j_1} G_{j_2}(t_{k-1}, x_{k-1}) |^2 ] \notag\\
            &~+ 2\mathbb{E} \left[ \langle e_{k-1}+(1 - \theta) (f(t_{k-1}, U_{t_{k-1}}) - f(t_{k-1}, u_{k-1})) \Delta, \mathcal{R}_k \rangle \right]  \notag\\
            &~+ 2\mathbb{E} \left[ \langle (g(t_{k-1}, U_{t_{k-1}}) - g(t_{k-1}, u_{k-1})) \Delta W_{k-1}, \mathcal{R}_k \rangle \right]  \notag\\
            &~+2\mathbb{E} \left[ \left\langle \sum_{j_1,j_2=1}^m  \big(A_{t_{k-1}}^-\widetilde{\mathcal{L}}^{j_1} G_{j_2}(t_{k-1}, X_{t_{k-1}})-A_{t_{k-1}}^-\widetilde{\mathcal{L}}^{j_1} G_{j_2}(t_{k-1}, x_{k-1}) \big)I_{j_1,j_2}^{t_{k-1},t_k}, \mathcal{R}_k \right\rangle \right].
      \end{align}
      Rearranging \eqref{FG-absG}, we obtain
      \begin{align}\label{FG-absH}
			 &~\mathbb{E} \left[
            \left| e_k - \theta \left( f(t_k, U_{t_k})
            - f(t_k, u_k) \right) \Delta \right|^2
            \right]
            \notag\\
            =&~ \mathbb{E} \left[ |e_{k-1} - \theta (f(t_{k-1}, U_{t_{k-1}}) 
            - f(t_{k-1}, u_{k-1})) \Delta|^2 \right] 
			+ 2\Delta \mathbb{E} \left[ \langle e_{k-1}, f(t_{k-1}, U_{t_{k-1}}) 
            - f(t_{k-1}, u_{k-1}) \rangle \right] \notag\\
			&~+ (1 - 2\theta) \Delta^2 \mathbb{E} \left[ |f(t_{k-1}, U_{t_{k-1}}) 
            - f(t_{k-1}, u_{k-1})|^2 \right] 
			+ \Delta \mathbb{E} \left[ |g(t_{k-1}, U_{t_{k-1}}) 
            - g(t_{k-1}, u_{k-1})|^2 \right] \notag\\
			&~+\mathbb{E} \left[ |\mathcal{R}_k|^2 \right]
            +\frac{\Delta^2}{2} \sum_{j_1, j_2=1}^m \mathbb{E} [ 
            |A_{t_{k-1}}^-\widetilde{\mathcal{L}}^{j_1} G_{j_2}(t_{k-1}, X_{t_{k-1}})
            - A_{t_{k-1}}^-\widetilde{\mathcal{L}}^{j_1} G_{j_2}(t_{k-1}, x_{k-1}) |^2 ] \notag\\
			&~+ 2\mathbb{E} \left[ \langle e_{k-1}+(1 - \theta) (f(t_{k-1}, U_{t_{k-1}}) 
            - f(t_{k-1}, u_{k-1})) \Delta, \mathcal{R}_k \rangle \right] \notag\\
			&~+ 2\mathbb{E} \left[ \langle (g(t_{k-1}, U_{t_{k-1}}) 
            - g(t_{k-1}, u_{k-1})) \Delta W_{k-1}, \mathcal{R}_k \rangle \right]\notag\\
			&~ +2\mathbb{E} \left[ \left\langle \sum_{j_1,j_2=1}^m  
            \big(A_{t_{k-1}}^-\widetilde{\mathcal{L}}^{j_1} G_{j_2}(t_{k-1},
            X_{t_{k-1}})-A_{t_{k-1}}^-\widetilde{\mathcal{L}}^{j_1} G_{j_2}(t_{k-1}, 
            x_{k-1}) \big)I_{j_1,j_2}^{t_{k-1},t_k}, \mathcal{R}_k \right\rangle \right] .
      \end{align}
      Since $e_{k-1}$, $f(t_{k-1},U_{t_{k-1}})$, and $f(t_{k-1},u_{k-1})$ are $\mathcal F_{t_{k-1}}$-measurable, Young's inequality and $\frac{1-\theta}{\theta}\leq1$ give
      \begin{align}\label{FG-absI}
            &2\mathbb{E}\left[ \left\langle e_{k-1}+(1 - \theta)(f(t_{k-1}, U_{t_{k-1}}) - f(t_{k-1}, u_{k-1}))\Delta, \mathcal{R}_k \right\rangle \right] \notag\\
            =&~ 2\mathbb{E}\left[ \left\langle e_{k-1}+(1 - \theta)(f(t_{k-1}, U_{t_{k-1}}) - f(t_{k-1}, u_{k-1}))\Delta, \mathbb{E}\left[\mathcal{R}_k | \mathcal{F}_{t_{k-1}}\right] \right\rangle \right] \notag\\
            =&~ \frac{2\theta - 2}{\theta} \mathbb{E}\left[ \left\langle e_{k-1} - \theta(f(t_{k-1}, U_{t_{k-1}}) - f(t_{k-1}, u_{k-1}))\Delta, \mathbb{E}\left[\mathcal{R}_k | \mathcal{F}_{t_{k-1}}\right] \right\rangle \right] \notag\\
            &~+ \frac{2}{\theta} \mathbb{E}\left[ \left\langle e_{k-1}, \mathbb{E}\left[\mathcal{R}_k | \mathcal{F}_{t_{k-1}}\right] \right\rangle \right] \notag\\
            \leq&~ \frac{1 - \theta}{\theta} \left( \Delta \mathbb{E}[|e_{k-1} - \theta(f(t_{k-1}, U_{t_{k-1}}) - f(t_{k-1}, u_{k-1}))\Delta|^2] \right)
            \notag\\
            &~+ \Delta^{-1} \mathbb{E}[|\mathbb{E}\left[\mathcal{R}_k | \mathcal{F}_{t_{k-1}}\right]|^2]+\Delta \mathbb{E}[|e_{k-1}|^2]  
           +\frac{1}{\theta^2} \Delta^{-1} \mathbb{E}[|\mathbb{E}\left[\mathcal{R}_k | \mathcal{F}_{t_{k-1}}\right]|^2] \notag\\
            \leq&~ \Delta \mathbb{E}[|e_{k-1} - \theta (f(t_{k-1}, U_{t_{k-1}}) - f(t_{k-1}, u_{k-1}))\Delta|^2]
            \notag\\
            &~+ \Delta \mathbb{E}[|e_{k-1}|^2]+\frac{\theta^2+1}{\theta^2} \Delta^{-1} \mathbb{E}[|\mathbb{E}\left[\mathcal{R}_k | \mathcal{F}_{t_{k-1}}\right]|^2].
      \end{align}
      Applying Young's inequality again, we have
      \begin{align}\label{FG-absJ1}
            &~2\mathbb{E}[\langle (g(t_{k-1}, U_{t_{k-1}}) - g(t_{k-1}, u_{k-1}))\Delta W_{k-1}, \mathcal{R}_k \rangle]  \notag\\
            \leq&~ (p_1 - 2)\Delta \mathbb{E}[|g(t_{k-1}, U_{t_{k-1}}) - g(t_{k-1}, u_{k-1})|^2]+\frac{1}{p_1 - 2}\mathbb{E}[|\mathcal{R}_k|^2],
      \end{align}
      and
      \begin{align}\label{FG-absJ2}
            &~2\mathbb{E} \left[ \big\langle \sum_{j_1,j_2=1}^m  \big(A_{t_{k-1}}^-\widetilde{\mathcal{L}}^{j_1} G_{j_2}(t_{k-1}, X_{t_{k-1}})-A_{t_{k-1}}^-\widetilde{\mathcal{L}}^{j_1} G_{j_2}(t_{k-1}, x_{k-1}) \big)I_{j_1,j_2}^{t_{k-1},t_k}, \mathcal{R}_k \big\rangle \right]  \notag\\
            \leq&~\frac{(q-1)\Delta^2}{2}\sum_{j_1, j_2=1}^m \mathbb{E} \big[ |A_{t_{k-1}}^-\widetilde{\mathcal{L}}^{j_1} G_{j_2}(t_{k-1}, X_{t_{k-1}})-A_{t_{k-1}}^-\widetilde{\mathcal{L}}^{j_1} G_{j_2}(t_{k-1}, x_{k-1}) |^2\big] +\frac{1}{q - 1}\mathbb{E}[|\mathcal{R}_k|^2].
      \end{align}
      Inserting \eqref{FG-absI}, \eqref{FG-absJ1} and \eqref{FG-absJ2} into \eqref{FG-absH}, using Assumption \ref{asm:FG},
      and noting that $\frac{1}{2} \leq \theta \leq 1$, we obtain
      \begin{align}\label{FG-absK}
            &\mathbb{E} \left[ \left| e_k - \theta \left( f(t_k, U_{t_k}) 
            - f(t_k, u_k) \right) \Delta \right|^2 \right] \notag\\
            \leq&~ \mathbb{E} \left[ \left| e_{k-1} 
            - \theta \left( f(t_{k-1}, U_{t_{k-1}}) 
            - f(t_{k-1}, u_{k-1}) \right) \Delta \right|^2 \right] \notag\\
            &~+ L_1 \Delta \mathbb{E} \left[ |X_{t_{k-1}}-x_{k-1}|^2 \right] 
            - (p_1-1) \Delta \mathbb{E} \left[ |g(t_{k-1}, U_{t_{k-1}}) 
            - g(t_{k-1}, u_{k-1})|^2 \right]  \notag\\
            &~-\frac{q\Delta_0^2}{2}\sum_{j_1, j_2=1}^m \mathbb{E} 
            [ |A_{t_{k-1}}^-\widetilde{\mathcal{L}}^{j_1} G_{j_2}(t_{k-1}, 
            X_{t_{k-1}})-A_{t_{k-1}}^-\widetilde{\mathcal{L}}^{j_1} 
            G_{j_2}(t_{k-1}, x_{k-1}) |^2]  \notag\\
            &~+ \Delta \mathbb{E} \left[ |g(t_{k-1}, U_{t_{k-1}}) 
            - g(t_{k-1}, u_{k-1})|^2 \right]
            + \mathbb{E} \left[ |\mathcal{R}_k|^2 \right]
            \notag\\
            &~+\frac{\Delta^2}{2} \sum_{j_1, j_2=1}^m \mathbb{E} 
            [ |A_{t_{k-1}}^-\widetilde{\mathcal{L}}^{j_1} G_{j_2}(t_{k-1}, 
            X_{t_{k-1}})-A_{t_{k-1}}^-\widetilde{\mathcal{L}}^{j_1} G_{j_2}(t_{k-1}, x_{k-1}) |^2 ]
            \notag\\
            &~+ \Delta \mathbb{E} \left[ \left| e_{k-1} 
            - \theta \left( f(t_{k-1}, U_{t_{k-1}}) 
            - f(t_{k-1}, u_{k-1}) \right) \Delta \right|^2 \right]
            + \Delta \mathbb{E} \left[ |e_{k-1}|^2 \right]
            \notag\\
            &~+ \frac{\theta^2+1}{\theta^2} \Delta^{-1} 
            \mathbb{E} \left[ \left| \mathbb{E}\left[\mathcal{R}_k 
            | \mathcal{F}_{t_{k-1}}\right] \right|^2 \right] 
            +(p_1 - 2) \Delta \mathbb{E} \left[ |g(t_{k-1}, 
            U_{t_{k-1}}) - g(t_{k-1}, u_{k-1})|^2 \right] 
            \notag\\
            &~+ \frac{1}{p_1 - 2} \mathbb{E} \left[ |\mathcal{R}_k|^2 \right]
            +\frac{(q-1)\Delta^2}{2}\sum_{j_1, j_2=1}^m 
            \mathbb{E} [ |A_{t_{k-1}}^-\widetilde{\mathcal{L}}^{j_1} G_{j_2}(t_{k-1}, X_{t_{k-1}}) \notag
            \\&~- A_{t_{k-1}}^-\widetilde{\mathcal{L}}^{j_1} G_{j_2}(t_{k-1}, x_{k-1}) |^2] 
            +\frac{1}{q - 1}\mathbb{E}[|\mathcal{R}_k|^2]. 
      \end{align}
      By Lemma \ref{lem:def}, $x_k\in\mathcal M_{t_k}$. Hence \eqref{def:Qx=v} gives $x_k = Px_k + \hat{V}(t_k,Px_k) = u_k+\hat{V}(t_k,u_k)$. Combining this representation with \eqref{lem-v-rs0}, we obtain, for any $k=0,1,\cdots,K$,
      \begin{align}
            \mathbb{E}\left[|X_{t_k}-x_k|^2\right]
            =&~
            \mathbb{E}\left[
            \left|
            \big(PX_{t_k}+\hat{V}(t_k,PX_{t_k})\big)
            -
            \big(Px_k+\hat{V}(t_k,Px_k)\big)
            \right|^2
            \right] \notag\\
            \leq&~
            (1+\hat L)^2
            \mathbb{E}\left[|PX_{t_k}-Px_k|^2\right]
            =
            (1+\hat L)^2
            \mathbb{E}\left[|e_k|^2\right].
            \label{FG-absAA}
      \end{align}
      By collecting like terms in \eqref{FG-absK} and combining \eqref{FG-absAA} and $0< \Delta \leq \bar{\Delta} \leq \Delta_0$, we have
      \begin{align*}
            &~\mathbb{E} \left[ \left| e_k - \theta \left( f(t_k, U_{t_k}) 
            - f(t_k, u_k) \right) \Delta \right|^2 \right] \notag\\
            \leq&~ (1+\Delta) \mathbb{E} \left[ \left| e_{k-1} 
            - \theta \left( f(t_{k-1}, U_{t_{k-1}}) 
            - f(t_{k-1}, u_{k-1}) \right) \Delta \right|^2 \right] 
            + \bigl(L_1(1+\hat{L})^2 +1\bigr) \Delta \mathbb{E} 
            \left[ |e_{k-1}|^2 \right]
            \\&~+
            \left(\frac{p_1 -1}{p_1 - 2}+\frac{1}{q - 1}\right) 
            \mathbb{E} \left[ |\mathcal{R}_k|^2 \right] 
            + 
            \frac{\theta^2+1}{\theta^2} \Delta^{-1} 
            \mathbb{E} \left[ \left| \mathbb{E}\left[\mathcal{R}_k 
            | \mathcal{F}_{t_{k-1}}\right] \right|^2 \right].
      \end{align*}
      
      Since $e_0=0$, $f(t_0,U_{t_0})-f(t_0,u_0)=0$, and $(1+\Delta)^k\leq e^T$ for $k=0,1,\cdots,K$, iteration gives
      \begin{align*}
            &\mathbb{E}[|e_k - \theta(f(t_k, U_{t_k}) - f(t_k, u_k))\Delta|^2]\\
            \leq&~ (1+\Delta)^k \mathbb{E} \left[ |e_0 - \theta(f(t_0, U_{t_0}) - f(t_0, u_0))\Delta|^2 \right]
           +(L_1(1+\hat{L})^2 +1) \Delta \sum_{i=0}^{k-1} (1+\Delta)^{k-1-i} \mathbb{E} \left[ |e_i|^2 \right]\\
            &~+\left( \frac{p_1 -1}{p_1 - 2}+ \frac{1}{q - 1}\right) \sum_{i=1}^k (1+\Delta)^{k-i} \mathbb{E} \left[ |\mathcal{R}_i|^2 \right]+\frac{\theta^2+1}{\theta^2} \Delta^{-1} \sum_{i=1}^k (1+\Delta)^{k-i} \mathbb{E} \left[ |\mathbb{E}\left[\mathcal{R}_i | \mathcal{F}_{t_{i-1}}\right]|^2 \right]\\
            \leq&~ (L_1(1+\hat{L})^2 +1) e^T \Delta \sum_{i=0}^{k-1} \mathbb{E} \left[ |e_i|^2 \right]+\left( \frac{p_1 -1}{p_1 - 2}+ \frac{1}{q - 1}\right)e^T \sum_{i=1}^k \mathbb{E} \left[ |\mathcal{R}_i|^2 \right]\\
            &~+ \frac{\theta^2+1}{\theta^2} e^T \Delta^{-1} \sum_{i=1}^k \mathbb{E} \left[ |\mathbb{E}\left[\mathcal{R}_i | \mathcal{F}_{t_{i-1}}\right]|^2 \right].
      \end{align*} 
      Lemma \ref{lem-v-3} gives
      \begin{align*}
            |e_k-\theta(f(t_k,U_{t_k})-f(t_k,u_k))\Delta|^2
            \geq&~
            |e_k|^2-2\theta\Delta\langle e_k,f(t_k,U_{t_k})-f(t_k,u_k)\rangle
            \\\geq&~
            (1-2L_1(1+\hat L^2)\theta\Delta)|e_k|^2,
      \end{align*}
      and consequently
      \begin{align*}
            (1 - 2L_{1}\big(1+\hat{L}^2\big)\theta \Delta) \mathbb{E}[|e_k|^2]
            \leq&~ (L_1(1+\hat{L})^2 +1) e^T \Delta \sum_{i=0}^{k-1} \mathbb{E}[|e_i|^2] 
            +
            \left( \frac{p_1-1}{p_1 - 2} + \frac{1}{q - 1}\right)
            \\&~\times 
            e^T \sum_{i=1}^K \mathbb{E}[|\mathcal{R}_i|^2] 
            + 
            \frac{\theta^2+1}{\theta^2} e^T \Delta^{-1} \sum_{i=1}^K 
            \mathbb{E}[|\mathbb{E}\left[\mathcal{R}_i | \mathcal{F}_{t_{i-1}}\right]|^2].
      \end{align*}
      Note that  $\Delta \leq \bar{\Delta} \leq \frac{1}{4\theta L_1(1+\hat L^2)}$, we have
      $1-2L_1(1+\hat L^2)\theta\Delta\geq\frac12$. The discrete Gronwall inequality therefore yields a constant $C>0$, independent of $\Delta$, such that for any $k=0,1,\cdots,K$,
      \begin{align*}
            \mathbb{E}[|e_k|^2] \leq C \left( \sum_{i=1}^K \mathbb{E}[|\mathcal{R}_i|^2]+\Delta^{-1} \sum_{i=1}^k \mathbb{E}[|\mathbb{E}\left[\mathcal{R}_i | \mathcal{F}_{t_{i-1}}\right]|^2]\right),
      \end{align*}
      which in combination with \eqref{FG-absAA} gives the desired result.
\end{proof}

\subsection{Global convergence rate}
The projection-invariant regularity estimates established in Lemmas \ref{lem:reduced-growth-estimates} and \ref{lem:reduced-derivative-difference-estimates} now provide the local bounds needed for the convergence theorem.

\begin{theorem}\label{thm:convAt}
Suppose that Assumptions \ref{asm:At}, \ref{ass:index1}, \ref{asm:FG}, and \ref{ass:time-regularity} hold. Then, for every $0<\Delta\leq\bar\Delta$, there exists a constant $C>0$, independent of $\Delta$, such that
      \begin{equation*}
      \max_{1\leq k\leq K} \mathbb{E}\left[|X_{t_k}-x_k|^2\right] \leq C\Delta^2 .
      \end{equation*}
\end{theorem}

\begin{proof}
We first decompose the local residual into the drift residual and the diffusion-Milstein residual. Write $\mathcal R_k=\mathcal R_k^f+\mathcal R_k^g$, where
      \begin{align*}
            \mathcal R_k^f
            :=&~
            \int_{t_{k-1}}^{t_k}
            \left(f(s,U_s) - \theta f(t_k,U_{t_k})
            -
            (1-\theta)f(t_{k-1},U_{t_{k-1}})\right)ds,
            \\
            \mathcal R_k^g
            :=&~
            \int_{t_{k-1}}^{t_k}
            \left( g(s,U_s)-g(t_{k-1},U_{t_{k-1}}) \right)dW_s
            -
            \sum_{j_1,j_2=1}^{m}
            A_{t_{k-1}}^-\widetilde{\mathcal L}^{j_1}G_{j_2}
            (t_{k-1},X_{t_{k-1}})I_{j_1,j_2}^{t_{k-1},t_k}.
      \end{align*}
      Hence
      \begin{equation}
            \mathbb E\left[|\mathcal R_k|^2\right] 
            \leq 
            C\mathbb E\left[|\mathcal R_k^f|^2\right]
            +
            C\mathbb E\left[|\mathcal R_k^g|^2\right]. \label{eq:Rk-ms-split}
      \end{equation}
      For the drift residual, by Jensen's inequality and \eqref{eq:red-ms-fg},
      \begin{align}
            \mathbb E\left[|\mathcal R_k^f|^2\right]
            \leq&~
            C\Delta\int_{t_{k-1}}^{t_k}
            \mathbb E\left[|f(s,U_s)-f(t_k,U_{t_k})|^2 \right]ds \notag
            +
            C\Delta
            \int_{t_{k-1}}^{t_k}\mathbb E\left[
            |f(s,U_s)-f(t_{k-1},U_{t_{k-1}})|^2 \right]ds \notag\\
            \leq&~
            C\Delta\int_{t_{k-1}}^{t_k}
            \left(|t_k-s|+|s-t_{k-1}|\right)ds
            \leq
            C\Delta^3 .
            \label{eq:Rk-f-est}
      \end{align}
      For a differentiable function $\phi(t,u):[0,T]\times\mathbb R^d\to\mathbb R^d$, we use the expansion
      \begin{align*}
            \phi(t,U_t)-\phi(s,U_s)
            =&~
            \partial_u\phi(s,U_s)(U_t-U_s)
            +
            \partial_t\phi(s,U_s)(t-s)
            +
            R_\phi(s,U_s,t,U_t),
            \quad 0\leq s<t\leq T,
      \end{align*}
where
      \begin{align*}
            R_\phi(s,U_s,t,U_t)
            :=&~
            \int_0^1
            \left(
            \partial_u\phi
            (s+r(t-s),U_s+r(U_t-U_s))
            -
            \partial_u\phi(s,U_s)
            \right)
            (U_t-U_s)\,dr
            \\
            &~+
            \int_0^1
            \left(
            \partial_t\phi
            (s+r(t-s),U_s+r(U_t-U_s))
            -
            \partial_t\phi(s,U_s)
            \right)
            (t-s)\,dr .
      \end{align*}
      Applying the above expansion with $\phi=g_j$, $s=t_{k-1}$ and $t\in[t_{k-1},t_k]$, we have
      \begin{align*}
            &~g_j(t,U_t)-g_j(t_{k-1},U_{t_{k-1}})
            \\=&~
            \partial_ug_j(t_{k-1},U_{t_{k-1}})
            (U_t-U_{t_{k-1}})
            +
            \partial_tg_j(t_{k-1},U_{t_{k-1}})(t-t_{k-1})
            +
            R_{g_j}(t_{k-1},U_{t_{k-1}},t,U_t).
      \end{align*}
      By \eqref{eq:inherent_U},
      \begin{equation*}
            U_t-U_{t_{k-1}} 
            = 
            \int_{t_{k-1}}^t f(\xi,U_\xi)\,d\xi
            +
            \int_{t_{k-1}}^t g(\xi,U_\xi)\,dW_\xi,
      \end{equation*}
      Substituting this formula into the expansion above and using \eqref{eq:reduced-milstein-coefficient}, we obtain \eqref{eq:diffusion-remainder-decomposition} for $t\in[t_{k-1},t_k]$:
      \begin{align}\label{eq:diffusion-remainder-decomposition}
            &~g_j(t,U_t)-g_j(t_{k-1},U_{t_{k-1}})
            -
            \sum_{j_1=1}^m
            \partial_ug_j
            (t_{k-1},U_{t_{k-1}})
            g_{j_1}(t_{k-1},U_{t_{k-1}})
            (W_t^{j_1}-W_{t_{k-1}}^{j_1})
            \notag\\
            =&~
            \partial_ug_j(t_{k-1},U_{t_{k-1}})
            \int_{t_{k-1}}^t f(\xi,U_\xi)\,d\xi
            +
            \partial_ug_j(t_{k-1},U_{t_{k-1}})
            \int_{t_{k-1}}^t
            \big(g(\xi,U_\xi)-g(t_{k-1},U_{t_{k-1}})\big)\,dW_\xi
            \notag\\
            &~+
            \partial_tg_j(t_{k-1},U_{t_{k-1}})
            (t-t_{k-1})
            +
            R_{g_j}(t_{k-1},U_{t_{k-1}},t,U_t).
      \end{align}
      Integrating \eqref{eq:diffusion-remainder-decomposition} with respect to $W^j$ over $[t_{k-1},t_k]$ and summing over $j=1,\cdots,m$ yields the desired decomposition of the diffusion part of $\mathcal R_k$. For the first term, by \eqref{eq:red-growth-gu}, \eqref{eq:red-growth-fg}, Lemma \ref{lem:bound}, H\"older's inequality and the fact $3\gamma-1<10\gamma-2<p_1$,
      \begin{align*}
            &~\mathbb{E}\left[
            \left|
            \partial_ug_j(t_{k-1},U_{t_{k-1}})
            \int_{t_{k-1}}^t f(\xi,U_\xi)\,d\xi
            \right|^2
            \right]
            \\
            \leq&~
            C(t-t_{k-1})
            \int_{t_{k-1}}^t
            \mathbb{E}\left[
            \left|
            \partial_ug_j(t_{k-1},U_{t_{k-1}})
            f(\xi,U_\xi)
            \right|^2
            \right]\,d\xi
                        \\
            \leq&~
            C(t-t_{k-1})
            \int_{t_{k-1}}^t
            \mathbb{E}\left[
            \left(
            1+|U_{t_{k-1}}|+|U_\xi|
            \right)^{3\gamma-1}
            \right]\,d\xi
            \\
            \leq&~
            C(t-t_{k-1})^2 .
      \end{align*}
      For the second term, by the It\^o isometry and \eqref{eq:red-growth-gu}, we have
      \begin{align*}
            &~\mathbb{E}\left[
            \left|
            \partial_ug_j(t_{k-1},U_{t_{k-1}})
            \int_{t_{k-1}}^{t}
            \left(
            g(\xi,U_\xi)-g(t_{k-1},U_{t_{k-1}})
            \right)dW_\xi
            \right|^2
            \right]
            \\
            =&~
            \int_{t_{k-1}}^{t}
            \mathbb{E}\left[
            \left|
            \partial_ug_j(t_{k-1},U_{t_{k-1}})
            \left(
            g(\xi,U_\xi)-g(t_{k-1},U_{t_{k-1}})
            \right)
            \right|^2
            \right]d\xi
            \\
            \leq&~
            C\int_{t_{k-1}}^{t}
            \mathbb{E}\left[
            (1+|X_{t_{k-1}}|)^{\gamma-1}
            \left|
            g(\xi,U_\xi)-g(t_{k-1},U_{t_{k-1}})
            \right|^2
            \right]d\xi .
      \end{align*}
      Furthermore,
      \begin{align*}
            g(\xi,U_\xi)-g(t_{k-1},U_{t_{k-1}})
            =&~(A_\xi^--A_{t_{k-1}}^-)G(\xi,X_\xi)
            +A_{t_{k-1}}^-[G(\xi,X_\xi)-G(t_{k-1},X_{t_{k-1}})].
      \end{align*}
      Since $2\gamma+2\gamma=4\gamma<p_1$, Corollary \ref{cor:mixed-moment}, with $a=2\gamma$ and $b=2$, and Lemmas \ref{lem:prop:A} and \ref{lem:bound}, together with \eqref{asm:FG-absG}, give
      \begin{align*}
            &~\mathbb{E}\left[
            (1+|X_{t_{k-1}}|)^{\gamma-1}
            \left|
            g(\xi,U_\xi)-g(t_{k-1},U_{t_{k-1}})
            \right|^2
            \right]
            \\
            \leq&~
            C\mathbb{E}\left[
            (1+|X_\xi|+|X_{t_{k-1}}|)^{2\gamma}
            |X_\xi-X_{t_{k-1}}|^2
            \right]
            +
            C|\xi-t_{k-1}|^2
            \mathbb{E}\left[
            (1+|X_\xi|+|X_{t_{k-1}}|)^{2\gamma}
            \right]
            \\
            \leq&~
            C|\xi-t_{k-1}|.
      \end{align*}
      Therefore,
      \begin{align*}
            &~\mathbb{E}\left[\left|\partial_ug_j(t_{k-1},U_{t_{k-1}})
            \int_{t_{k-1}}^{t}\left(g(\xi,U_\xi)-g(t_{k-1},U_{t_{k-1}})
            \right)dW_\xi\right|^2\right]
            \leq
            C\int_{t_{k-1}}^{t}|\xi-t_{k-1}|\,d\xi
            \leq
            C(t-t_{k-1})^2 .
      \end{align*}
      For the third term, by \eqref{eq:red-growth-gu}, Lemma \ref{lem:bound} and the fact $2\gamma-2<p_1$,
      \begin{align*}
            \mathbb{E}\left[\left|\partial_tg_j(t_{k-1},U_{t_{k-1}})
            (t-t_{k-1})\right|^2\right]
            \leq
            C(t-t_{k-1})^2 \mathbb{E}\left[
            \left( 1+|U_{t_{k-1}}| \right)^{2\gamma-2}\right]
            \leq
            C(t-t_{k-1})^2 .
      \end{align*}
      Note that for any $s,\bar{s}\in[t_{k-1},t_k]$, along the line segment joining $(s,U_s)$ and $(\bar{s},U_{\bar{s}})$, the estimates \eqref{eq:red-diff-gu}--\eqref{eq:red-diff-gt}, together with $X_s=\Psi_{s}(U_s)$, $U_s=PX_s$ and Lemma \ref{lem-v}, yield, uniformly in $r\in[0,1]$,
       \begin{align}
            \big|\Psi_{s+r(\bar{s}-s)}\big(U_s+r(U_{\bar{s}}-U_s)\big)\big|
            =&~
            \big| \big(U_s+r(U_{\bar{s}}-U_s)\big) 
            +
            \hat{V} \left( s+r(\bar{s}-s), \big(U_s+r(U_{\bar{s}}-U_s)\big) \right) \big| \notag 
            \\\leq&~
            C\big(1+\left| P \big(U_s+r(U_{\bar{s}}-U_s)\big) \right|\big)
            \leq
            C\bigl(1+|PX_s|+|PX_{\bar{s}}|\bigr) \notag 
            \\\leq&~
            C\bigl(1+|X_s|+|X_{\bar{s}}|\bigr),
            \label{eq:Psi:1}
      \end{align}
      and
      \begin{align}
            \left|\Psi_{s+r(\bar{s}-s)}
            \big(U_s+r(U_{\bar{s}}-U_s)\big)
            -X_s\right|
            =&~
            \left|\Psi_{s+r(\bar{s}-s)}
            \big(U_s+r(U_{\bar{s}}-U_s)\big)
            - \Psi_s (U_s) \right| \notag \\
            \leq &~
            (1+\hat{L})r|P(U_{\bar{s}}-U_s)|
            +C\bigl(1+|PU_s|+|PU_{\bar{s}}|\bigr)^\gamma |\bar{s}-s| \notag \\
            \leq &~
            C|X_s-X_{\bar{s}}|
            +C\bigl(1+|X_s|+|X_{\bar{s}}|\bigr)^\gamma \Delta.
            \label{eq:Psi:2}
      \end{align}
      These together with \eqref{eq:red-diff-gu}--\eqref{eq:red-diff-gt} yields
      \begin{align}
            &~ \left| \partial_u g_j
            \big(s+r(\bar{s}-s),U_s+r(U_{\bar{s}}-U_s)\big)
            -
            \partial_u g_j (s,U_s) \right|^2 \notag \\
            \leq &~
             C \big(1+\left|\Psi_{s+r(\bar{s}-s)}\big(U_s+r(U_{\bar{s}}-U_s)\big)\right|+|X_s|\big)^{3\gamma-1}
             \big( \left| \Psi_{s+r(\bar{s}-s)}\big(U_s+r(U_{\bar{s}}-U_s)\big) -X_s \right|^2 +\Delta^2 \big) \notag \\
            \leq &~
            C \big(1+|X_s|+|X_{\bar{s}}| \big)^{3\gamma-1}  
            \big( |X_s-X_{\bar{s}}|^2+\big(1+|X_s|+|X_{\bar{s}}|\big)^{2\gamma} \Delta^2 +\Delta^2 \big)\notag \\
            \leq &~
            C \big(1+|X_s|+|X_{\bar{s}}| \big)^{3\gamma-1} |X_s-X_{\bar{s}}|^2
            +C \Delta^2 \big(1+|X_s|+|X_{\bar{s}}| \big)^{5\gamma-1},
            \label{eq:Psi:gu}
      \end{align}      
      and
            \begin{align}
            &~ \left| \partial_t g_j
            \big(s+r(\bar{s}-s),U_s+r(U_{\bar{s}}-U_s)\big)
            -
            \partial_t g_j (s,U_s)\right|^2 \notag \\
            \leq &~
             C \big(1+\left|\Psi_{s+r(\bar{s}-s)}\big(U_s+r(U_{\bar{s}}-U_s)\big)\right|+|X_s|\big)^{7\gamma-3}
             \big( \left| \Psi_{s+r(\bar{s}-s)}\big(U_s+r(U_{\bar{s}}-U_s)\big) -X_s \right|^2 +\Delta^2 \big) \notag \\
            \leq &~
            C \big(1+|X_s|+|X_{\bar{s}}| \big)^{7\gamma-3}  
            \big( |X_s-X_{\bar{s}}|^2+\bigl(1+|X_s|+|X_{\bar{s}}|\bigr)^{2\gamma} \Delta^2 +\Delta^2 \big)\notag \\
            \leq &~
            C \big(1+|X_s|+|X_{\bar{s}}| \big)^{7\gamma-3} |X_s-X_{\bar{s}}|^2
            +C \Delta^2 \big(1+|X_s|+|X_{\bar{s}}| \big)^{9\gamma-3}.
            \label{eq:Psi:gt}
      \end{align}
      Similarly, using \eqref{eq:red-diff-fu} and \eqref{eq:red-diff-ft} yields
            \begin{align}
            &~ \left| \partial_u f
            \big(s+r(\bar{s}-s),U_s+r(U_{\bar{s}}-U_s)\big)
            -
            \partial_u f (s,U_s) \right|^2 \notag \\
            \leq &~
            C \big(1+|X_s|+|X_{\bar{s}}| \big)^{4\gamma} |X_s-X_{\bar{s}}|^2
            +C \Delta^2 \big(1+|X_s|+|X_{\bar{s}}| \big)^{6\gamma},
            \label{eq:Psi:fu}
      \end{align}      
      and
            \begin{align}
            &~ \left| \partial_t f
            \big(s+r(\bar{s}-s),U_s+r(U_{\bar{s}}-U_s)\big)
            -
            \partial_t f (s,U_s)\right|^2 \notag \\
            \leq &~
            C\big(1+|X_s|+|X_{\bar{s}}| \big)^{8\gamma-2} |X_s-X_{\bar{s}}|^2
            +
            C\Delta^2 \big(1+|X_s|+|X_{\bar{s}}| \big)^{10\gamma-2}.
            \label{eq:Psi:ft}
      \end{align}
      By setting $s=t_{k-1} $ and $\bar{s}=t$ in \eqref{eq:Psi:gu}--\eqref{eq:Psi:gt}, we obtain the following estimate:
      \begin{align*}
            &~ |R_{g_j}(t_{k-1},U_{t_{k-1}},t,U_t)|^2
            \\\leq&~
            2 \int_0^1
            \big| \left(
            \partial_u g_j
            \big( t_{k-1}+r(t-t_{k-1}),U_{t_{k-1}}+r(U_t-U_{t_{k-1}}) \big)
            -
            \partial_u g_j(t_{k-1},U_{t_{k-1}})
            \right)(U_t-U_{t_{k-1}})\big|^2 \,dr
            \\
            &~+
            2 \int_0^1
            \big| \left(
            \partial_t g_j
            (t_{k-1}+r(t-t_{k-1}),U_{t_{k-1}}+r(U_t-U_{t_{k-1}}))
            -
            \partial_t g_j(t_{k-1},U_{t_{k-1}})
            \right)(t-t_{k-1})\big|^2 \,dr 
            \\\leq&~
            C (1+|X_t|+|X_{t_{k-1}}|)^{3\gamma-1}|X_t-X_{t_{k-1}}|^4
            +
            C\Delta^2 (1+|X_t|+|X_{t_{k-1}}|)^{5\gamma-1} |X_t-X_{t_{k-1}}|^2 
            \\&~+
            C\Delta^2 (1+|X_t|+|X_{t_{k-1}}|)^{7\gamma-3}|X_t-X_{t_{k-1}}|^2
            +
            C\Delta^4 (1+|X_t|+|X_{t_{k-1}}|)^{9\gamma-3}.
      \end{align*}
      Since $7\gamma-1\leq 9\gamma-3<p_1$, Corollary \ref{cor:mixed-moment} and Lemma \ref{lem:bound} yield
      \begin{align*}
            \mathbb{E}\left[|R_{g_j}(t_{k-1},U_{t_{k-1}},t,U_t)|^2\right]
            \leq~
            C\Delta^2 + C\Delta^3 + C\Delta^3 + C\Delta^4
            \leq~
            C\Delta^2 .
      \end{align*}

      Hence, by \eqref{eq:diffusion-remainder-decomposition},
      \begin{align}
            \mathbb{E}[|\mathcal R_k^g|^2]
            \leq&~
            C\int_{t_{k-1}}^{t_k}
            \Delta^2\,dt
            \leq
            C\Delta^3 .
            \label{eq:Rk-g-est}
      \end{align}
      Combining \eqref{eq:Rk-ms-split}, \eqref{eq:Rk-f-est}, and \eqref{eq:Rk-g-est}, we obtain
      \begin{equation}\label{eq:local-Rk-L2}
            \mathbb{E}\left[|\mathcal R_k|^2\right]
            \leq
            C\Delta^3 .
      \end{equation}
      The stochastic-integral and Milstein iterated-integral parts both have zero conditional expectation with respect to $\mathcal F_{t_{k-1}}$. Hence only the two drift residuals remain. Since $\frac12\leq\theta\leq1$,
      \begin{align*}
            &~\mathbb{E}\left[\left|\mathbb{E}\left[\mathcal R_k|\mathcal F_{t_{k-1}}\right]\right|^2\right]
            \\=&~
            \mathbb{E}\left[\left|\mathbb{E}\left[
            \theta\int_{t_{k-1}}^{t_k} \big(f(s,U_s)-f(t_k,U_{t_k})\big)\,ds
            +(1-\theta) \int_{t_{k-1}}^{t_k}\big(f(s,U_s)-f(t_{k-1},U_{t_{k-1}})\big)\,ds
            \,\big|\,\mathcal F_{t_{k-1}}\right]\right|^2\right]
            \\\leq&~
            2\mathbb{E}\left[\left|\mathbb{E}\left[\int_{t_{k-1}}^{t_k}
            \big(f(s,U_s)-f(t_k,U_{t_k})\big)\,ds\,\big|\,\mathcal F_{t_{k-1}}
            \right]\right|^2\right]
            \\&~+
            2\mathbb{E}\left[\left|\mathbb{E}\left[\int_{t_{k-1}}^{t_k}
            \big(f(s,U_s)-f(t_{k-1},U_{t_{k-1}})\big)\,ds
            \,\big|\,\mathcal F_{t_{k-1}}\right]\right|^2\right].
      \end{align*}
      Applying the above expansion with $\phi=f$, for $s\in[t_{k-1},t_k]$ we get
      \begin{align*}
            f(t_k,U_{t_k})-f(s,U_s)
            =&~
            \partial_uf(s,U_s)(U_{t_k}-U_s)
            +
            \partial_tf(s,U_s)(t_k-s)
            +
            R_f(s,U_s,t_k,U_{t_k}).
      \end{align*}
      The required integrability follows from \eqref{eq:red-growth-fu}, \eqref{eq:red-growth-fg}, and Lemma \ref{lem:bound}. More precisely,
      \begin{equation*}
            \int_{t_{k-1}}^{t_k}\int_s^{t_k} 
            \mathbb E\left[ \left| \partial_uf(s,U_s) g(\xi,U_\xi) \right|^2 \right]
            \,\mathrm d\xi\,\mathrm ds < \infty.
      \end{equation*}
      Hence Fubini's theorem for conditional expectations, the tower property, and the martingale property of the It\^o integral give
      \begin{align*}
            &~\mathbb E\left[
            \int_{t_{k-1}}^{t_k}\partial_uf(s,U_s)
            \int_s^{t_k}g(\xi,U_\xi)\,\mathrm dW_\xi\,\mathrm ds
            \,\middle|\,\mathcal F_{t_{k-1}}\right]\\
            =&~\int_{t_{k-1}}^{t_k}\mathbb E\left[
            \int_s^{t_k}\partial_uf(s,U_s)g(\xi,U_\xi)\,\mathrm dW_\xi
            \,\middle|\,\mathcal F_{t_{k-1}}\right]\mathrm ds\\
            =&~\int_{t_{k-1}}^{t_k}\mathbb E\left[
            \mathbb E\left[
            \int_s^{t_k}\partial_uf(s,U_s)g(\xi,U_\xi)\,\mathrm dW_\xi
            \,\middle|\,\mathcal F_s\right]
            \,\middle|\,\mathcal F_{t_{k-1}}\right]\mathrm ds\\
            =&~0.
      \end{align*}
      Therefore, using the expansion above, the $L^2$-contraction property of conditional expectation, and Minkowski's inequality,
      \begin{align}
            &~\Bigg(\mathbb{E}\left[
            \left|
            \mathbb{E}\left[
            \int_{t_{k-1}}^{t_k}
            \big(f(s,U_s)-f(t_k,U_{t_k})\big)\,ds
            \,\big|\,\mathcal F_{t_{k-1}}
            \right]
            \right|^2
            \right]\Bigg)^{1/2}
            \notag\\
            \leq&~
            \int_{t_{k-1}}^{t_k}
            \int_s^{t_k}
            \bigg(\mathbb{E}\left[
            \left|
            \partial_uf(s,U_s)
            f(\xi,U_\xi)
            \right|^2
            \right]\bigg)^{1/2}
            d\xi ds
            +
            \int_{t_{k-1}}^{t_k}
            \bigg(\mathbb{E}\left[
            \left|
            \partial_tf(s,U_s)(t_k-s)
            \right|^2
            \right]\bigg)^{1/2}
            ds
            \notag\\
            &~+
            \int_{t_{k-1}}^{t_k}
            \bigg(\mathbb{E}\left[
            |R_f(s,U_s,t_k,U_{t_k})|^2
            \right]\bigg)^{1/2}
            ds .
            \label{eq:first-drift-cond-root}
      \end{align}
      By \eqref{eq:red-growth-fu}, \eqref{eq:red-growth-fg}, and Lemma \ref{lem:bound}, the first two terms on the right-hand side satisfy
      \begin{gather*}
            \int_{t_{k-1}}^{t_k}\int_s^{t_k}\Bigg(\mathbb{E}\left[
            \left|\partial_uf(s,U_s)f(\xi,U_\xi)\right|^2
            \right]\Bigg)^{1/2}d\xi ds
            \leq
            C\int_{t_{k-1}}^{t_k}\int_s^{t_k}d\xi ds
            \leq C\Delta^2,
            \\
            \int_{t_{k-1}}^{t_k}\Bigg(\mathbb{E}\left[
            \left|\partial_tf(s,U_s)(t_k-s)
            \right|^2\right]\Bigg)^{1/2}ds
            \leq
            C\int_{t_{k-1}}^{t_k}(t_k-s)ds
            \leq 
            C\Delta^2 .
      \end{gather*}
      By setting $\bar{s}=t_k$ in \eqref{eq:Psi:fu}--\eqref{eq:Psi:ft}, we obtain the following estimate:
      \begin{align*}
            &~ |R_f(s,U_s,t_k,U_{t_k})|^2
            \\\leq&~
            2 \int_0^1 \big| \left( \partial_u f
            \big( s+r({t_k}-s),U_s+r(U_{t_k}-U_s) \big)
            -
            \partial_u f (s,U_s)\right) (U_{t_k}-U_s)\Big|^2 \,dr
            \\&~+
            2 \int_0^1 \big| \left( \partial_t f
            (s+r({t_k}-s),U_s+r(U_{t_k}-U_s))
            -
            \partial_t f (s,U_s)\right) ({t_k}-s)\big|^2 \,dr
            \\\leq&~
            C (1+|X_t|+|X_{t_{k-1}}|)^{4\gamma} |X_t-X_{t_{k-1}}|^4 
            +
            C\Delta^2 (1+|X_t|+|X_{t_{k-1}}|)^{6\gamma} |X_t-X_{t_{k-1}}|^2
            \\&~+
            C\Delta^2 (1+|X_t|+|X_{t_{k-1}}|)^{8\gamma-2} |X_t-X_{t_{k-1}}|^2
            +
            C\Delta^4 (1+|X_t|+|X_{t_{k-1}}|)^{10\gamma-2}.
      \end{align*}
      Since $8\gamma \leq 10\gamma-2<p_1$, Corollary \ref{cor:mixed-moment} and Lemma \ref{lem:bound} yield
      \begin{align*}
            &~\mathbb{E}\left[
            |R_f(s,U_s,t_k,U_{t_k})|^2
            \right]
            \leq
            C\Delta^2+C\Delta^3
            +C\Delta^3+C\Delta^4
            \leq
            C\Delta^2 .
      \end{align*}
      Here, the last inequality follows from $0<\Delta\leq\bar\Delta$. Combining this bound with \eqref{eq:first-drift-cond-root} and the two preceding estimates gives
      \begin{equation*}
      \mathbb{E}\left[ \left| \mathbb{E}\left[ \int_{t_{k-1}}^{t_k} \big(f(s,U_s)-f(t_k,U_{t_k})\big)\,ds \,\big|\,\mathcal F_{t_{k-1}} \right] \right|^2 \right] \leq C\Delta^4 .
      \end{equation*}
      The drift residual frozen at $t_{k-1}$ is treated in the same way. For $s\in[t_{k-1},t_k]$,
      \begin{align*}
            f(s,U_s)-f(t_{k-1},U_{t_{k-1}})
            =&~
            \partial_uf(t_{k-1},U_{t_{k-1}})
            (U_s-U_{t_{k-1}})
            +
            \partial_tf(t_{k-1},U_{t_{k-1}})
            (s-t_{k-1})
            \\
            &~+
            R_f(t_{k-1},U_{t_{k-1}},s,U_s).
      \end{align*}
      Repeating the estimates above, with $t_k-s$ replaced by $s-t_{k-1}$, gives
      $\mathbb{E}\left[ |R_f(t_{k-1},U_{t_{k-1}},s,U_s)|^2 \right] \leq C(s-t_{k-1})^2$
      and hence
      \begin{align*}
            &~\mathbb{E}\left[
            \left|
            \mathbb{E}\left[
            \int_{t_{k-1}}^{t_k}
            \big(f(s,U_s)-f(t_{k-1},U_{t_{k-1}})\big)\,ds
            \,\big|\,\mathcal F_{t_{k-1}}
            \right]
            \right|^2
            \right]
            \leq C\Delta^4 .
      \end{align*}
      Combining the two drift conditional estimates and using the zero conditional expectation of the stochastic-integral and iterated-integral parts of \eqref{eq:local-residual}, we obtain
      \begin{equation}\label{eq:local-Rk-cond}
            \mathbb{E}\left[
            \left|
            \mathbb{E}\left[\mathcal R_k|\mathcal F_{t_{k-1}}\right]
            \right|^2
            \right]
            \leq C\Delta^4 .
      \end{equation}
      Combining \eqref{eq:local-Rk-L2}, \eqref{eq:local-Rk-cond}, and Lemma \ref{lem:error-reduction}, we obtain
      \begin{align*}
            \max_{1\leq k\leq K}
            \mathbb{E}\left[|X_{t_k}-x_k|^2\right]
            \leq&~
            C\left(
            \sum_{k=1}^K \Delta^3
            +
            \Delta^{-1}\sum_{k=1}^K \Delta^4
            \right)
            \leq
            C\Delta^2.
      \end{align*}
      This completes the proof.
\end{proof}

\section{Numerical experiments}\label{sec:experiments}
This section presents numerical experiments illustrating the theoretical results. For the convergence experiment, the nonlinear equation at each time step is solved by the Newton--Raphson method with tolerance $10^{-5}$. Since the exact solution is unavailable, the reference solution is generated by \eqref{eq:ABEM} with the fine stepsize $\Delta=2^{-13}$, while the numerical approximations use $\Delta=2^{-i}$ for $i=7,8,9,10,11$. The expectations are estimated from $1000$ Brownian sample paths. To illustrate constraint preservation, we report the Monte Carlo root mean square constraint residual at each time level:
\begin{equation*}
      \operatorname{Res}(t_k) := \left(\frac{1}{M}\sum_{m=1}^{M} \big|RF(t_k,x_k^{(\Delta,m)})\big|^2\right)^{1/2}, \quad k = 0, 1, \cdots, K.
\end{equation*}

Consider the following $\mathbb R^3$-valued index-$1$ SDAE:
\begin{equation}\label{eq:example}
      A_t\diff X_t=F(t,X_t)\diff t+G(t,X_t)\diff W_t, \quad t\in[0,T],
\end{equation}
where $\{W_t\}_{t\in[0,T]}$ is a standard $\mathbb R^3$-valued Brownian motion. Let
\begin{align*}
      &A_t=
      \begin{pmatrix}
            1+t & 0 & 0 \\
            0 & 2+t & 0 \\
            0 &  0& 0
      \end{pmatrix}.
\end{align*}
This matrix is singular, and Assumption \ref{asm:At} holds with $r=2$, $\underline{\sigma}=1$, and $\overline{\sigma}=2+T$. Direct calculation gives
\begin{align*}
      A_t^- &=
      \begin{pmatrix}
            \frac{1}{1+t} & 0 & 0 \\
            0 & \frac{1}{2+t} & 0 \\
            0 & 0 & 0
      \end{pmatrix},
      \quad
      P =
      \begin{pmatrix}
            1 & 0 & 0 \\
            0 & 1 & 0 \\
            0 & 0 & 0
      \end{pmatrix},
      \quad
      R =
      \begin{pmatrix}
            0 & 0 & 0 \\
            0 & 0 & 0 \\
            0 & 0 & 1
      \end{pmatrix}.
\end{align*}

For $x\in\mathbb R^3$ and $t\in[0,T]$, let
\begin{equation*}
      F(t,x)=\left((1+t)(-x_1-x_1^3-x_1^5), (2+t)(-x_2-x_2^3-x_2^5),x_1+x_2+x_3\right)^\top,
\end{equation*}
\begin{align*}
      &G(t,x)=
      \begin{pmatrix}
            0.06(1+t)x_1^2 & 0 & 0 \\
            0 & 0.05(2+t)x_2^2 & 0 \\
            0 & 0 & 0
      \end{pmatrix},
\end{align*}
with columns
\begin{equation*}
      G_1(t,x)=(0.06(1+t)x_1^2,0,0)^\top, \quad G_2(t,x)=(0,0.05(2+t)x_2^2,0)^\top, \quad G_3(t,x)=(0,0,0)^\top.
\end{equation*}
The required derivatives are
\begin{align*}
      &\partial_xF(t,x) =
      \begin{pmatrix}
            (1+t)(-1-3x_1^2-5x_1^4) & 0 & 0 \\
            0 & (2+t)(-1-3x_2^2-5x_2^4) & 0 \\
            1 & 1 & 1
      \end{pmatrix} ,  
\end{align*}
\begin{align*}
      &\partial_tF(t,x)=(-x_1-x_1^3-x_1^5,-x_2-x_2^3-x_2^5,0)^\top,
\end{align*}
\begin{align*}
      &\partial_xG_1(t,x) = 
      \begin{pmatrix}
            0.12(1+t)x_1 & 0 & 0 \\
            0 & 0 & 0 \\
            0 & 0 & 0
      \end{pmatrix},
      \partial_xG_2(t,x) = 
      \begin{pmatrix}
            0 & 0 & 0 \\
            0 & 0.1(2+t)x_2 & 0 \\
            0 & 0 & 0
      \end{pmatrix},
      \partial_xG_3(t,x) = 
      \begin{pmatrix}
            0 & 0 & 0 \\
            0 & 0 & 0 \\
            0 & 0 & 0
      \end{pmatrix},
\end{align*}
and $\partial_tG_1(t,x)=(0.06x_1^2,0,0)^\top, \partial_tG_2(t,x)=(0,0.05x_2^2,0)^\top, \partial_tG_3(t,x)=(0,0,0)^\top$.

Noting that
\begin{align*}
      J(t, x) = A_t+R\partial_xF(t, x) = 
     \begin{pmatrix}
            1+t & 0 & 0 \\
            0 & 2+t & 0 \\
            1 & 1 & 1
      \end{pmatrix}, 
      \quad x \in \mathbb{R}^3, \, t \in [0, T],
\end{align*}
we have
$J(t, x)^{-1} =
      \begin{pmatrix}
            \frac{1}{1+t} & 0 & 0 \\
            0 & \frac{1}{2+t} & 0 \\
            -\frac{1}{1+t} & -\frac{1}{2+t} & 1
      \end{pmatrix}$.
Consequently, the only nonzero Milstein coefficients are
\begin{align*}
      \widetilde{\mathcal{L}}^{1}G_{1}(t,x)
      =
      \begin{pmatrix}
            2(0.06)^2(1+t)x_1^3\\0\\0
      \end{pmatrix},
      \quad
      \widetilde{\mathcal{L}}^{2}G_{2}(t,x)
      =
      \begin{pmatrix}
            0\\2(0.05)^2(2+t)x_2^3\\0
      \end{pmatrix}.
\end{align*}
All cross terms $\widetilde{\mathcal{L}}^{j_1}G_{j_2}$ with $j_1\neq j_2$, as well as all terms involving $j_1=3$ or $j_2=3$, vanish. In particular,
\begin{align*}
      A_t^-\widetilde{\mathcal{L}}^{1}G_{1}(t,x)
      =\bigl(2(0.06)^2x_1^3,0,0\bigr)^\top,
      \quad
      A_t^-\widetilde{\mathcal{L}}^{2}G_{2}(t,x)
      =\bigl(0,2(0.05)^2x_2^3,0\bigr)^\top.
\end{align*}
Assume that $0.06^2(p_1-1)-\frac32\leq0$. Using $m^2+mn+n^2\geq\frac34(m+n)^2$, $mn\leq m^2+n^2$, and $m^4+m^3n+m^2n^2+mn^3+n^4\geq\frac12(m^4+n^4)$ for $m,n\in\mathbb R$, we verify the coupled monotonicity condition for $x,y\in\mathbb R^3$, $t\in[0,T]$, $0<\Delta\leq\Delta_0$ with $\Delta_0=1$, and $q=2$. First, we have
\begin{align*}
      &2\big\langle Px-Py, 
      A_t^- F(t,x) - A_t^- F(t,y) \big\rangle \\
      =&~
      -2\big( x_1-y_1 \big)^2-\ 2 \big( x_1-y_1 \big) \big( x_1^5-y_1^5 \big)-\ 2 \big( x_1-y_1 \big) \big( x_1^3-y_1^3 \big)-2\big( x_2-y_2 \big)^2\\
      &~-\ 2\big( x_2-y_2 \big) \big( x_2^5-y_2^5 \big)  -\ 2\big( x_2-y_2 \big) \big( x_2^3-y_2^3 \big)\\
      \leq&~ - 2 \big( x_1-y_1 \big)^2 \big( x_1^4 +x_1^3 y_1+x_1^2 y_1^2 +x_1 y_1^3+ y_1^4 \big)
      -\ 2 \big( x_2-y_2 \big)^2 \big( x_2^4 +x_2^3 y_2+x_2^2 y_2^2 +x_2 y_2^3+ y_2^4 \big)
      \\
      &~-2 \big( x_1-y_1 \big)^2\big(x_1^2+x_1 y_1 +y_1^2\big)-2 \big( x_2-y_2 \big)^2\big(x_2^2+x_2 y_2 +y_2^2\big)
      +2\big( x_1-y_1 \big)^2+2\big( x_2-y_2 \big)^2\\ 
      \leq&~\ -  \big( x_1-y_1 \big)^2 \big( x_1^4+y_1^4 \big)
      -\  \big( x_2-y_2 \big)^2 \big( x_2^4+y_2^4 \big)-\frac{3}{2}\big( x_1-y_1 \big)^2\big( x_1+y_1 \big)^2 
      \\
      &~-\frac{3}{2}\big( x_2-y_2 \big)^2\big( x_2+y_2 \big)^2+2\big( x_1-y_1 \big)^2+2\big( x_2-y_2 \big)^2,
\end{align*}
and
\begin{align*}
      &(p_1-1)|A_t^- G(t,x) - A_t^- G(t,y)|^2+\frac{q}{2}\Delta\sum_{j_1,j_2=1}^3 \big|A_t^-\widetilde{\mathcal{L}}^{j_1}G_{j_2}(t,x) - A_t^-\widetilde{\mathcal{L}}^{j_1}G_{j_2}(t,y)\big|^2\\
      \leq&~
      (p_1-1) \left( 0.06^2\big( x_1^2-y_1^2 \big)^2+0.05^2\big( x_2^2-y_2^2 \big)^2  \right)+ 0.06^4 4\big( x_1^3-y_1^3 \big)^2+0.05^4 4\big( x_2^3-y_2^3 \big)^2\\
      =&~(p_1-1) \left( 0.06^2\big( x_1-y_1 \big)^2 \big( x_1+y_1 \big)^2+0.05^2\big( x_2-y_2 \big)^2 \big( x_2+y_2 \big)^2 \right) \\
      &~+ 0.06^4 4\big( x_1-y_1 \big)^2\big( x_1^2+x_1 y_1 +y_1^2 \big)^2+0.05^4 4\big( x_2-y_2 \big)^2\big( x_2^2 +x_2 y_2 +y_2^2 \big)^2\\
      \leq&~ 
      (p_1-1) \left( 0.06^2\big( x_1-y_1 \big)^2 \big( x_1+y_1 \big)^2+0.05^2\big( x_2-y_2 \big)^2 \big( x_2+y_2 \big)^2 \right) \\
      &~+ 0.06^4 32\big( x_1-y_1 \big)^2\big( x_1^4 +y_1^4 \big)+0.05^4 32\big( x_2-y_2 \big)^2\big( x_2^4  +y_2^4 \big).
\end{align*}

Combining the preceding estimates, we obtain
\begin{align*}
      &2\big\langle Px-Py, 
      A_t^- F(t,x) - A_t^- F(t,y) \big\rangle 
      +
      (p_1-1)|A_t^- G(t,x) - A_t^- G(t,y)|^2\notag
      \\
      &~+
      \frac{q}{2}\Delta\sum_{j_1,j_2=1}^3 \big|A_t^-\widetilde{\mathcal{L}}^{j_1}G_{j_2}(t,x) - A_t^-\widetilde{\mathcal{L}}^{j_1}G_{j_2}(t,y)\big|^2\notag
      \\
      \leq&~ (0.06^4 32 -1)\big( x_1-y_1 \big)^2\big( x_1^4 +y_1^4 \big)+(0.05^4 32 -1)\big( x_2-y_2 \big)^2\big( x_2^4  +y_2^4 \big)\\
      &~+\big(0.06^2(p_1-1)-\frac{3}{2}\big)\big( x_1-y_1 \big)^2\big( x_1+y_1 \big)^2
      +\big(0.05^2(p_1-1)-\frac{3}{2}\big)\big( x_2-y_2 \big)^2\big( x_2+y_2 \big)^2\\
      &~+2\big( x_1-y_1 \big)^2+2\big( x_2-y_2 \big)^2\\
      \leq&~\ 2 |x-y|^2.
\end{align*}

\begin{figure}[!htbp]
\begin{center}
      \subfigure[$\theta=0.5$]
      {\includegraphics[width=0.3\textwidth]{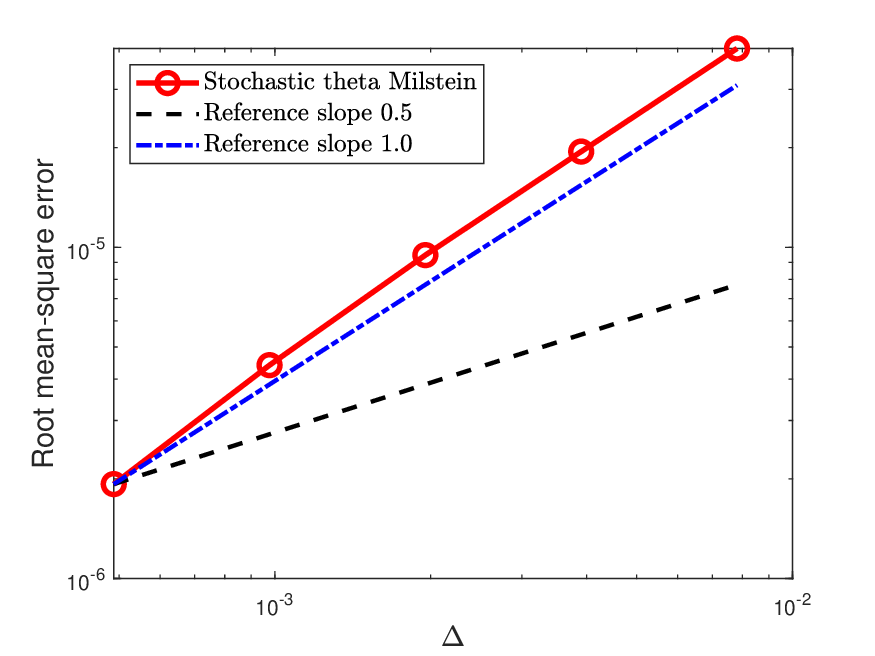}}
      \quad
      \subfigure[$\theta=0.75$]
      {\includegraphics[width=0.3\textwidth]{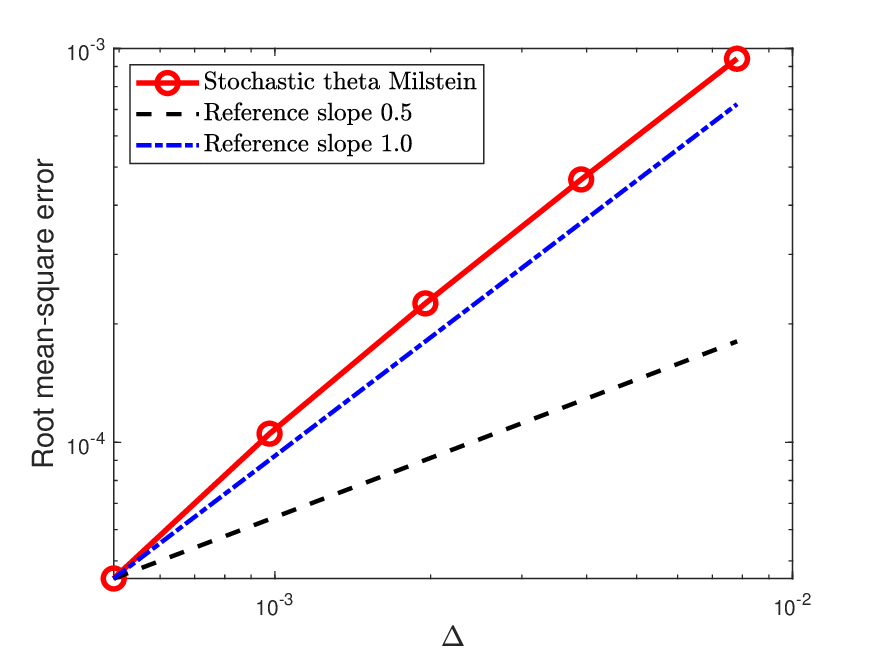}}
      \quad
      \subfigure[$\theta=1$]
      {\includegraphics[width=0.3\textwidth]{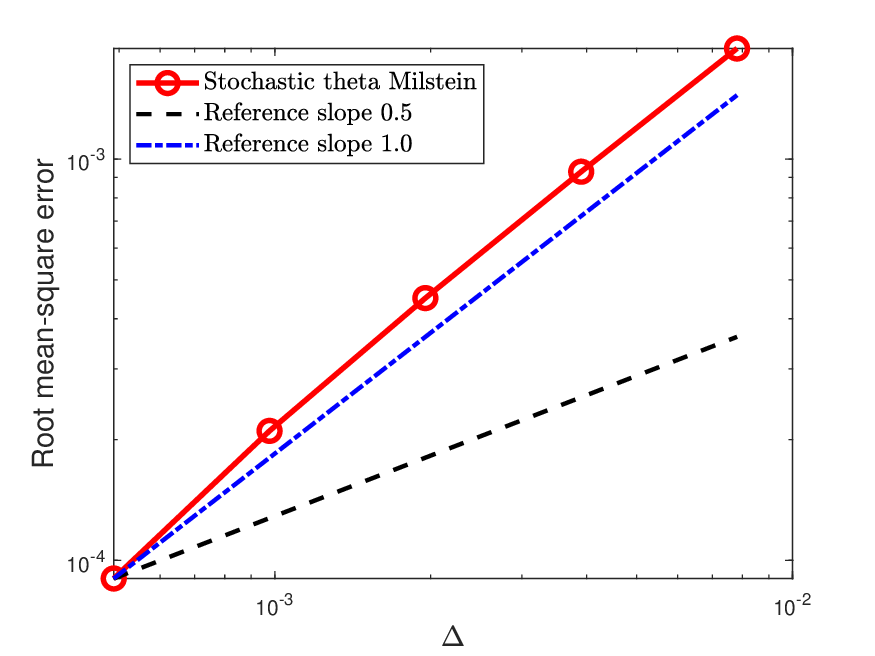}}
      \caption{Root mean square errors of the stochastic theta Milstein method applied to \eqref{eq:example}}
      \label{fig:three}
\end{center}
\end{figure}

For the derivatives of $F$, direct calculation gives
\begin{align*}
      \left| \left( \partial_xF(t,x) 
      - \partial_xF(s,y) \right)  \right|
      =&~\big(\big((1+t)(-1-3x_1^2-5x_1^4)-(1+s)(-1-3y_1^2-5y_1^4)\big)^2\\
      &~+\big((2+t)(-1-3x_2^2-5x_2^4)-(2+s)(-1-3y_2^2-5y_2^4)\big)^2\big)^\frac{1}{2}\\
      \leq&~ C\left((1+|x|+|y|)^3 |x-y|+(1+|x|+|y|)^5|t-s|\right) , 
\end{align*}
and
\begin{align*}
      &\left| \left( \partial_tF(t,x) 
      - \partial_tF(s,y) \right)  \right| \\
      =&~\big(\big(-x_1-x_1^3-x_1^5+y_1+y_1^3+y_1^5\big)^2+\big(-x_2-x_2^3-x_2^5+y_2+y_2^3+y_2^5\big)^2\big)^\frac{1}{2}\\
      \leq&~ C\left((1+|x|+|y|)^4 |x-y|+(1+|x|+|y|)^5|t-s|\right). 
\end{align*}
Similarly, for every $j=1,2,3$,
\begin{align*}
      \left| \left( \partial_xG_j(t,x) 
      - \partial_xG_j(s,y) \right)  \right|^2
      \leq C\left((1+|x|+|y|)^2 |x-y|^2+(1+|x|+|y|)^6|t-s|^2\right), 
\end{align*}
and
\begin{align*}
      \left| \left( \partial_tG_j(t,x) 
      - \partial_tG_j(s,y) \right)  \right|^2
      \leq C\left((1+|x|+|y|)^4 |x-y|^2+(1+|x|+|y|)^6|t-s|^2\right).
\end{align*}
Moreover,
\begin{align*}
      |J(t, x)^{-1}| = \left(\frac{2}{(1+t)^2}+\frac{2}{(2+t)^2} 
     + 1\right)^{1/2} \leq 4.
\end{align*}

The displayed formulas show that $F$ and $G_j$, $j=1,2,3$, are polynomials in the spatial variables with smooth time-dependent factors. Hence they belong to $C^{1,2}([0,T]\times\mathbb R^3)$, and the derivative estimates in Assumption \ref{ass:time-regularity} follow directly from the standard polynomial factorizations. Together with the preceding bounds, this verifies Assumptions \ref{asm:FG} and \ref{ass:time-regularity} with $\gamma=5$, $q=2$, and an initial choice $\Delta_0=1$. Letting $X_0=(0.5,0.4,-0.9)^\top$ gives $RF(0,X_0)=0$, and hence the initial-value condition in Assumption \ref{ass:index1} is satisfied. For the convergence experiment we take $T=1$ and $p_1=49$. Indeed,
\begin{equation*}
      p_1=49>10\gamma-2=48, \quad (0.06)^2(p_1-1)-\frac32 =0.0036\times48-1.5=-1.3272<0.
\end{equation*}

Thus all coefficient and regularity assumptions of Theorem \ref{thm:convAt} are satisfied. With $\bar\Delta$ defined by \eqref{eq:stepsize-threshold}, the theorem applies for every $0<\Delta\leq\bar\Delta$. This restriction is a sufficient theoretical condition and is not asserted to be sharp for the finite stepsizes used below. In Figure \ref{fig:three}, the root mean square error curves are nearly parallel to the reference line of slope one, numerically supporting strong order one in the root mean square norm.

\begin{figure}[!htbp]
\begin{center}
      \subfigure[$T=50$]{\includegraphics[width=0.4\textwidth]{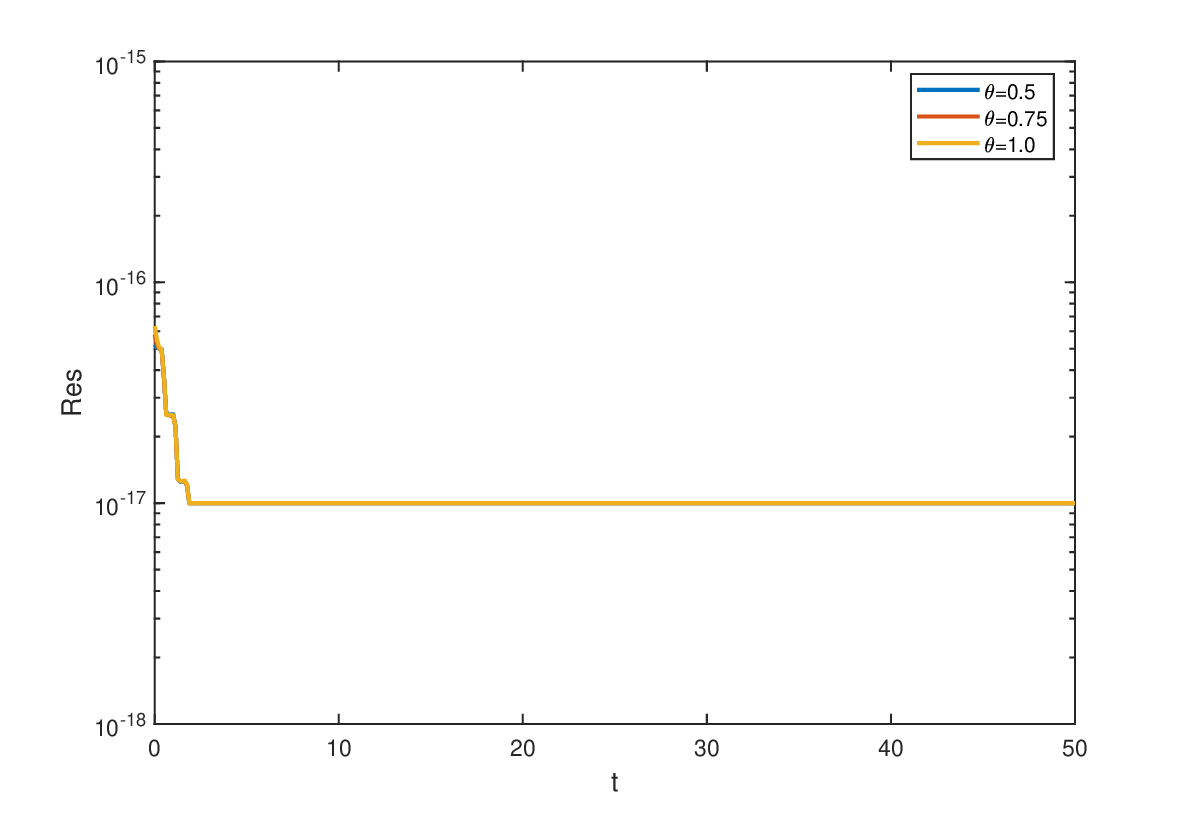}}
      \quad
      \subfigure[$T=100$]{\includegraphics[width=0.4\textwidth]{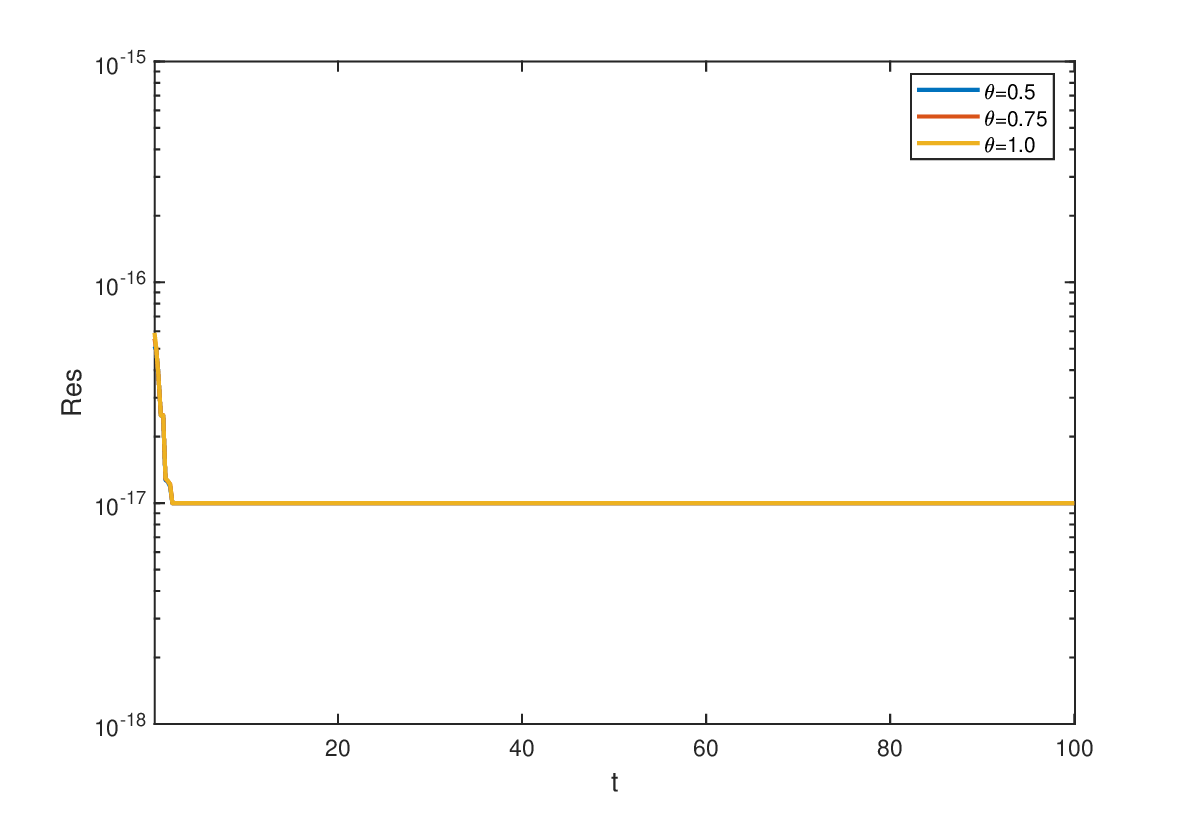}}
      \caption{Structure-preserving simulations for $T=50$ and $T=100$}
      \label{fig:constraint-residual-longtime}
\end{center}
\end{figure}

Figure \ref{fig:constraint-residual-longtime} illustrates the constraint preserving property of the stochastic theta Milstein method. Panels (a) and (b) show the Monte Carlo root mean square residuals of the algebraic constraint for $T=50$ and $T=100$, respectively. These simulations use $\Delta=2^{-8}$ and a Newton tolerance of $10^{-5}$. For every tested value of $\theta$, the residuals remain close to round-off level throughout the simulations, confirming constraint preservation over long time intervals.


\section{Conclusion and future work}

In this paper, we studied the strong convergence of structure-preserving stochastic theta Milstein methods for index-$1$ SDAEs with time-dependent singular matrices and non-globally Lipschitz coefficients. Under a fixed differential algebraic splitting, the algebraic-differential decomposition was used to formulate the method directly in the original SDAE variables and to identify its Milstein correction with that of the reduced SDE. We proved that, for $\theta\in[1/2,1]$ and sufficiently small stepsizes, the method is well posed and preserves the algebraic constraint at every time level. By combining an error-reduction argument with local Milstein expansions and polynomial moment estimates, we further established strong order one in the root mean square norm. The numerical experiments confirmed both the theoretical convergence order and the constraint preserving property, including over long time intervals.

Future work will focus on two directions. The first is to develop higher-order structure-preserving numerical methods for index-$1$ SDAEs with time-dependent singular matrices and non-globally Lipschitz coefficients, and to establish strong convergence orders higher than one. This will require identifying higher-order stochastic correction terms in the original SDAE variables and ensuring their compatibility with the algebraic constraints. The second is to study more general time-dependent singular matrices for which the differential and algebraic subspaces also vary with time. In this setting, the resulting time-dependent projectors and their derivatives introduce additional terms into the reduced dynamics, making both the construction of constraint preserving numerical methods and the corresponding weak convergence analysis substantially more challenging.

\bibliographystyle{plain}  
\bibliography{SDAEsMilsteinrefs}
\end{document}